\documentclass[12pt]{article}

\usepackage[T1]{fontenc}
\usepackage{amsmath}
\usepackage{amssymb}
\usepackage{latexsym}
\usepackage{mathrsfs}
\RequirePackage[raiselinks=false,colorlinks=true,citecolor=blue,urlcolor=blue,linkcolor=blue,bookmarksopen=true]{hyperref}

\usepackage[all]{xy}
\SelectTips{cm}{}
\newdir{ >}{{}*!/-7pt/@{>}}
\entrymodifiers={+!!<0pt,\fontdimen22\textfont2>}
\newcommand{\ulpullback}[1][ul]{\save*!/#1-1.2pc/#1:(-1,1)@^{|-}\restore}
\newcommand{\dlpullback}[1][dl]{\save*!/#1-1.2pc/#1:(-1,1)@^{|-}\restore}

\newcommand{\drpullback}[1][dr]{\save*!/#1-1.2pc/#1:(-1,1)@^{|-}\restore}

\usepackage[noBBpl]{mathpazo}
\renewcommand{\texttt}[1]{{\fontfamily{pcr}\fontseries{m}\fontshape{n}%
\selectfont #1}}

\newcommand{\kat}[1]{\text{\textbf{\textsl{#1}}}}

\makeatletter
\renewcommand{\ldots}{\relax\ifmmode\ldotp\ldotp\ldotp\else$\m@th\ldotp\ldotp\ldotp\ $\fi}
\makeatother

\providecommand{\qed}{\hspace*{\fill}\nolinebreak[2]\hspace*{\fill}$\Box$}
\renewcommand{\epsilon}{\varepsilon}
\newcommand{\wtil}{\widetilde}
\newcommand{\into}			{\hookrightarrow}
\newcommand{\isopil}{\stackrel{\raisebox{0.1ex}[0ex][0ex]{\(\sim\)}}%
			{\raisebox{-0.15ex}[0.28ex]{\(\rightarrow\)}}}

\newcommand{\Hom}	{\operatorname{Hom}}
\newcommand{\Aut}	{\operatorname{Aut}}

\newcommand{\Fun}	{\operatorname{Fun}}
\newcommand{\Sub}	{\operatorname{Sub}}

\newcommand{\shortsetminus}	{\,\raisebox{1pt}{\ensuremath{\mathbb r}\,}}

\newcommand{\ov}{\overline}
\newcommand{\un}{\underline}

\newcommand{\CC}{\mathscr C}

\providecommand{\overskrift}[1]{\par\noindent\relax{\LARGE #1}\par\bigskip}

\newcommand{\hovedfont}{\normalfont\bfseries}

\usepackage{theorem}
\theorembodyfont{\normalfont\itshape}
\theoremheaderfont{\hovedfont}
	\theoremstyle{change}
\newtheorem{lemma}{Lemma.}[subsection]
\newtheorem{prop}[lemma]{Proposition.}
\newtheorem{thm}[lemma]{Theorem.}
\newtheorem{cor}[lemma]{Corollary.}
	\theorembodyfont{\normalfont}
\newtheorem{eks}[lemma]{Example.}
\newtheorem{BM}[lemma]{Remark.}

\newtheorem{taller}[lemma]{$\!\!$}
\newenvironment{blanko}[1]%
{\begin{taller}{\hovedfont #1}\normalfont}%
{\end{taller}}
{%
\begin{list}{\em Definition. }%
{\setlength{\labelsep}{0mm}\setlength{\leftmargin}{0mm}%
\setlength{\labelwidth}{0mm}\setlength{\listparindent}{\parindent}%
\setlength{\parsep}{\parskip}\setlength{\partopsep}{0mm}}%
\item%
}%
{%
\end{list}%
}
\newenvironment{proof}%
{%
\begin{list}{\em Proof. }%
{\setlength{\labelsep}{0mm}\setlength{\leftmargin}{0mm}%
\setlength{\labelwidth}{0mm}\setlength{\listparindent}{\parindent}%
\setlength{\parsep}{\parskip}\setlength{\partopsep}{0mm}}%
\item%
}%
{%
\qed\end{list}%
}
\newenvironment{proof*}[1]%
{%
\begin{list}{\em #1 }%
{\setlength{\labelsep}{0mm}\setlength{\leftmargin}{0mm}%
\setlength{\labelwidth}{0mm}\setlength{\listparindent}{\parindent}%
\setlength{\parsep}{\parskip}\setlength{\partopsep}{0mm}}%
\item%
}%
{%
\qed\end{list}%
}
\newenvironment{blanko*}[1]%
{%
\begin{list}{\bf {#1} }%
{\setlength{\labelsep}{0mm}\setlength{\leftmargin}{0mm}%
\setlength{\labelwidth}{0mm}\setlength{\listparindent}{\parindent}%
\setlength{\parsep}{\parskip}\setlength{\partopsep}{0mm}}%
\item%
}%
{%
\end{list}%
}

\providecommand{\lastUpdate}[1]{#1}

\newcommand{\op}{^{\text{{\rm{op}}}}}

\DeclareMathOperator*{\colim}{colim}

\newcommand{\Sh}{\kat{Sh}}
\newcommand{\PrSh}{\kat{PrSh}}
\newcommand{\Grpd}{\kat{Grpd}}
\newcommand{\Set}{\kat{Set}}
\newcommand{\et}{\mathrm{et}}
\newcommand{\cancov}{\mathrm{cancov}}

\newcommand{\triv}{\,\pmb \shortmid\,}
\newcommand{\comma}{\raisebox{1pt}{$\downarrow$}}
\newcommand{\name}[1]{\ulcorner #1 \urcorner}

\newcommand{\Cor}{\kat{Cor}}
\newcommand{\res}{\operatorname{res}}
\newcommand{\iso}{\operatorname{iso}}
\newcommand{\Et}{\kat{Et}}
\newcommand{\el}{\operatorname{el}}
\newcommand{\id}{\operatorname{id}}

\newcommand{\compl}[1]{#1^{{\scriptscriptstyle{\complement}}}}

\usepackage{texdraw}
\everytexdraw{%
	\drawdim pt \linewd 0.5 \textref h:C v:C
	\setunitscale 1
}

\makeatletter
\renewcommand{\ps@headings}
	{\setlength{\headheight}{41pt}%
	 \setlength{\headsep}{12pt}%
	 \renewcommand{\@oddhead}{\parbox{\textwidth}{%
			\small
			\texttt{\jobname.tex \ \ \ \lastUpdate{2015-04-03 22:16}
			\hfill [\thepage/\pageref{lastpage}]}
			\\ \rule[8pt]{\textwidth}{0.3pt}}%
	 }
	\renewcommand{\@oddfoot}{}
	\renewcommand{\@evenfoot}{}%
}  
\makeatother

\makeatletter

\renewcommand*\l@section[2]{%
  \ifnum \c@tocdepth >\z@
    \addpenalty\@secpenalty
    \addvspace{0.4em \@plus\p@}
    \setlength\@tempdima{1.5em}%
    \begingroup
      \parindent \z@ \rightskip \@pnumwidth
      \parfillskip -\@pnumwidth
      \leavevmode \bfseries
      \advance\leftskip\@tempdima
      \hskip -\leftskip
      #1\nobreak\hfil \nobreak\hb@xt@\@pnumwidth{\hss #2}\par
    \endgroup
  \fi}

\renewcommand{\tableofcontents}{%
   \begin{center}
\begin{minipage}{12cm}
   \begin{center}
     
     \vspace{8pt}
     
       \bf{\contentsname}
   \end{center}
	
       \vspace{-18pt}
	
   \footnotesize
   \begin{center}
\@starttoc{toc}
   \end{center}	
\end{minipage}
	\end{center}
	\addvspace{3em \@plus\p@}
}

\makeatother

\newcommand{\elGr}{\kat{elGr}}

\newcommand{\Gr}{\kat{Gr}}
\newcommand{\HGr}{\kat{HGr}}

\newcommand{\GR}{\kat{Gr}^+}

\newcommand{\Map}{\operatorname{Map}}

\newcommand{\wfrac}[2]{\begin{array}{c}{#1}\\ \hline\hline {#2}\end{array}}

\begin{document}

\setcounter{secnumdepth}{2}

\pagestyle{headings}

\vspace*{8pt}

\pagestyle{headings}

\vspace*{8pt}

\begin{center}

\overskrift{Graphs, hypergraphs, and properads}

\vspace{0.7\bigskipamount}

\noindent
\textsc{Joachim Kock}
\footnote{
Departament de Matem\`atiques,
Universitat Aut\`onoma de Barcelona,
Spain
}

\texttt{kock@mat.uab.cat}

\end{center}

\begin{abstract}
  A categorical formalism for directed graphs is introduced, featuring natural
  notions of morphisms and subgraphs, and leading to two elementary descriptions
  of the free-properad monad, first in terms of presheaves on elementary graphs,
  second in terms of groupoid-enriched hypergraphs.
%
\end{abstract}


\footnotesize

\tableofcontents

\normalsize

\section*{Introduction}

\addcontentsline{toc}{section}{\numberline{}Introduction}

Properads were introduced by Vallette~\cite{Vallette:0411542}, as a notion
intermediate between operads and props, featuring important notions
and results generalised from operads, such as Koszul duality. 
From a combinatorial viewpoint, properads are to connected acyclic directed
graphs (henceforth just called {\em graphs}) as operads are to rooted trees.  

A combinatorial approach to coloured properads and infinity-properads has been
developed by Hackney, Robertson and Yau~\cite{Hackney-Robertson-Yau}, based on a
somewhat elaborate notion of graph (due to Yau and Johnson~\cite{Yau-Johnson}),
whose notions of morphism and subgraph are derived from properad notions, in
turn defined in terms of an operation of graph substitution (also from~\cite{Yau-Johnson}).

The present work (which grew out of studying \cite{Hackney-Robertson-Yau})
proposes a different approach to the relationship between
graphs and properads, in which the starting point is a graph formalism featuring
natural notions of morphisms and subgraphs, and featuring colimits enough to
describe the free-properad monad in terms of presheaves on elementary (directed)
graphs, and to prove a nerve theorem, in close analogy with the approach to
graphs and coloured modular operads (compact symmetric 
multicategories) of Joyal-Kock~\cite{Joyal-Kock:0908.2675}.
A second feature of the graph formalism introduced is that it naturally extends to
hypergraphs, and neatly explains the dual role of graphs as carriers of algebraic
structures (\ref{dual}).

The theory is developed from scratch (finite sets, pullbacks, colimits),
and follows the case of operads \cite{Kock:0807.2874} as closely as
possible.

In the case of operads, there is a natural category of trees and tree embeddings
\cite{Kock:0807.2874}, with a subcategory of elementary trees, such that {\em
(coloured) collections} (the structure underlying coloured operads) are
precisely presheaves on elementary trees, or equivalently sheaves on trees.  The
free-operad monad is given by a simple colimit formula exploiting this
equivalence of categories.  The free operad on a tree is not again a tree, but
one discovery of \cite{Kock:0807.2874} is that nevertheless it {\em is} represented
by the same shape
$$
A\leftarrow E \to B \to A.
$$
This shape is that of {\em polynomial endofunctors}, and the 
free-operad monad restricts to the free-monad monad on 
polynomial endofunctors, where it has a direct combinatorial 
description in terms of these representing diagrams \cite{Kock:0807.2874}.

The same features are shared by the case of properads: a natural
category $\Gr$ of (connected, acyclic, directed) graphs is introduced, with
a subcategory $\elGr$ of \mbox{elementary} graphs, such that {\em (coloured) bi-collections}
(the structure
underlying coloured properads) are precisely presheaves on $\elGr$, or
equivalently sheaves on $\Gr$.  Again the free-properad monad is given by a
simple colimit formula exploiting this equivalence of categories.  This time,
however, the category of graphs $\Gr$ involves {\em etale maps} instead of just
embeddings, and the notion of sheaves is with respect to the etale topology.
This is a crucial difference: in contrast to embeddings, etale maps have
automorphisms (deck transformations), and for this reason, when trying to mimic
the representability feature, {\em groupoids} are required, instead of sets.
With this proviso, the analogy from trees goes through: while the free properad
on a graph is not again a graph, it {\em is} a diagram of the same shape (now in
groupoids), and this shape,
$$
A \leftarrow I \to N \leftarrow O \to A
$$
is that of {\em (groupoid-enriched) hypergraphs} (hypergraphs with `stacky'
nodes).  Again, the free-properad monad restricts to a monad on such
hypergraphs, and has a direct combinatorial interpretation in terms of these
representing diagrams: while a hypergraph is given by its elementary subgraphs
(or more precisely, by etale maps from elementary graphs), the free properad on
it is given by etale maps from arbitrary (connected) graphs.

\pagebreak

In summary, the main notions involved fit into the following
schematic relationship:

\begin{center}
  \small
 \begin{tabular}{ |c|c|c|c|c| }
    \hline
    elementary~tree & tree & polynomial endo. & presheaf~on elem.~trees &operad
    \\ 
    \hline
        elementary~graph & \phantom{x}graph\phantom{x} & hypergraph & presheaf~on elem.~graphs &
	\phantom{i}properad\phantom{i} \\
	\hline
\end{tabular}
\end{center}

The category $\Gr$ encodes `geometric' aspects of the combinatorics of graphs
--- open inclusions, etale maps, symmetries, colimits.
Again in analogy with the case of
operads and trees, the free-properad monad generates a bigger category of graphs
$\wtil\Gr$, whose new maps, the graph refinements, embody `algebraic' aspects
--- basically substitution (see \ref{sec:generic}).  This bigger category
$\wtil\Gr$ is featured in a nerve theorem (\ref{thm:nerve}), characterising
properads among presheaves on $\wtil\Gr$ in terms of a Segal condition.  The
category $\wtil\Gr$ is shown to have a weak factorisation system given by
refinements and etale maps.  Cutting down the right-hand class from etale maps
to convex open inclusions results in a smaller category, which is the one first
constructed by Hackney, Robertson and Yau~\cite{Hackney-Robertson-Yau}.

Some of the results in Subsections~\ref{sec:freeproperad} and \ref{sec:wtilGr}
have some overlap with results in
Hackney-Robertson-Yau~\cite{Hackney-Robertson-Yau}, as indicated in each case.
The reader is strongly encouraged to follow these references to compare with a
different approach with its own advantages.

\section{Graphs}

\subsection{Graphs}

In this work, the word `graph' means `directed graph with open-ended
edges' (and from Section~\ref{sec:properads} and onwards, graphs will be assumed
connected and acyclic).  We proceed to give the formal definition, whose merit is
the elegant way morphisms and subgraphs are encoded.
  All constructions take place in the category of finite sets.
  When numbers are used as sets, they denote a set with that many elements.

\begin{blanko}{Definition.}\label{def:graph}
  A {\em graph} is a diagram of finite sets
  \begin{equation}\label{eq:graph}
    \xymatrix{
    A & \ar[l]_s I \ar[r]^p & N & \ar[l]_q O \ar[r]^t & A
    }
  \end{equation}
  for which $s$ and $t$ are injective. 
  
  Throughout, for any individual graph under consideration,
  we shall use these symbols to refer to its constituents, if no confusion seems 
  likely.
\end{blanko}

\begin{blanko}{Terminology and interpretation.}
  The set $A$ is the set of {\em edges} (`A' for `arc' or `ar\^ete').  The set
  $N$ is the set of {\em nodes}.  The set $I$ is the set of {\em in-flags}, and
  the set $O$ is the set of {\em out-flags}.  The maps $s$ and $t$ return the
  edge in a flag, and the maps $p$ and $q$ return the node in a flag.  Saying
  that $s$ is injective means that every edge is the incoming edge of at most
  one node, and similarly injectivity of $t$ means that every edge is the
  outgoing edge of at most one node.

  An edge $a\in A$ is called an {\em inner edge} if it belongs to the
  intersection $O \cap I = O \times_A I \subset A$.  The complement is called
  the set of {\em ports}.  The complement of $s$, i.e.~the set of edges which
  are not incoming edges of any node, is called the set of {\em exports}.  The
  complement of $t$ is called the set of {\em imports}.
\end{blanko}

\begin{blanko}{Unit graph.}
  The graph with one edge and without nodes is given by
  $$
  \xymatrix{
  1 & \ar[l] 0 \ar[r] & 0 & \ar[l] 0 \ar[r] & 1 .
  }
  $$
  This edge is simultaneously an import and an export (indeed the unique edge is 
  neither in the image of $s$, nor in the image of $t$).
  It is called the {\em unit graph}, denoted $U$. 
  
  Note that the edge itself has no sense of direction.  The notion of direction
  in a graph is provided only by the nodes, owing to the distinction made
  between in-flags and out-flags.  (Compare with categories: an object has no
  sense of direction; arrows (operations) do have a direction, expressed by
  source and target.  Further explanation of this analogy is provided by the 
  ambient category of hypergraphs, cf.~\ref{dual} below.)
\end{blanko}

\begin{blanko}{Corollas.}
   The {\em corolla} with $m$ imports and $n$ exports, denoted $C^m_n$,
   is the graph with one node given by
$$
\xymatrix{
m+n & \ar[l] m \ar[r] & 1 & \ar[l] n \ar[r] & m+n ,
}
$$
with the outer maps the obvious sum inclusions.
As a special case we have the
graph $C^0_0$ with one node and no edges, given by
$$
\xymatrix{
0 & \ar[l] 0 \ar[r] & 1 & \ar[l] 0 \ar[r] & 0 .
}
$$
\end{blanko}

\begin{blanko}{Wheels.}
  The graph $W_1$ with one node and one `loop' (cf.~\ref{loops})
  is given by
$$
\xymatrix{
1 & \ar[l] 1 \ar[r] & 1 & \ar[l] 1 \ar[r] & 1 .
}
$$
More generally, the {\em wheel} of length $n\geq 1$, is the
graph denoted $W_n$ given by
$$
\xymatrix{
n & \ar[l]_s n \ar[r]^p & n & \ar[l]_q n \ar[r]^t & n ,
}
$$
for which all the structure maps $s,t,p,q$ are bijections, and such that
the composite bijection $t \circ q^{-1} \circ p \circ s^{-1} : n \isopil n$ is a
cyclic permutation of $n$.
\end{blanko}

\begin{blanko}{Philosophical remarks.}
    From the viewpoint of properads, the nodes in a graph represent
    operations with multiple in- and outputs.  From this viewpoint,
    it is natural to try to define graphs as pairs of multi-valued maps
    
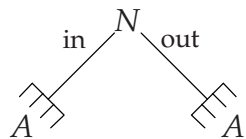
\begin{center}
    \begin{texdraw} 
	\htext (0 0) {$N$}
	\htext (-40 -40) {$A$}
	\htext (40 -40) {$A$}
	
	\move (-5 -5) \bsegment
 \lvec (-25 -25) \move (-31 -19) \lvec ( -19 -31) \lvec (-24 -36) \move ( -23 -27)
\lvec (-28 -32) \move (-27 -23) \lvec ( -32 -28) \move ( -31 -19) 
\lvec (-36 -24)
\esegment
	\move (5 -5) \bsegment
 \lvec (25 -25) \move (31 -19) \lvec (19 -31) \lvec (24 -36) \move (23 -27)
\lvec (28 -32) \move (27 -23) \lvec (32 -28) \move (31 -19) 
\lvec (36 -24)
\esegment
\htext (-20 -7){\footnotesize in}
\htext (20 -7){\footnotesize out}
\end{texdraw}
    \end{center}
 A standard way  to encode a multi-valued map is as a span.
 Hence we arrive at the shape \eqref{eq:graph}.
 
 On the other hand, since a closed directed graph is an endospan $E 
 \rightrightarrows V$ (source and target of an edge),  a directed
 graph admitting open-ended edges should be the same but with just partially defined maps.
 A standard way to encode a partially defined map is as a span in 
 which the backward arrow is an injection, hence again we arrive at 
 our shape of diagrams for a graph.
 
 These dual viewpoints also point towards the natural relationship
 with hypergraphs: the shape is naturally the juxtaposition of the
 incidence relations a-hyperedge-being-the-input-of-a-node and
 a-hyperedge-being-the-output-of-a-node, which is one way to represent
 directed hypergraphs, cf.~\ref{hyp} below.
 
 A main feature and motivation for the present graph implementation
 are the elegant notions of
 morphisms that follow from the definition.
\end{blanko}

\begin{blanko}{Morphisms.}
  A {\em morphism of graphs} is a commutative diagram
\begin{equation}\label{eq:map}
\xymatrix{
A' \ar[d]_\alpha& \ar[l] I'\ar[d] \ar[r] & N'\ar[d] & \ar[l] O'\ar[d] 
\ar[r] & A'\ar[d]^\alpha \\
A & \ar[l] I \ar[r] & N & \ar[l] O \ar[r] & A
}   
\end{equation}
Graphs and morphisms of graphs form a category denoted $\GR$.

Note that a morphism sends edges to edges and nodes to nodes, respecting the
incidence relations.  In particular it sends inner edges to inner edges.
Ports are not necessarily sent to ports: a
port may be sent to an inner edge.

A morphism is {\em etale} if the two squares in the middle are 
pullbacks:
$$
\xymatrix{
A' \ar[d]_\alpha& \ar[l] I'\ar[d] \drpullback\ar[r] & N'\ar[d]^\varphi & 
\ar[l]\dlpullback O'\ar[d] 
\ar[r] & A'\ar[d]^\alpha \\
A & \ar[l] I \ar[r] & N & \ar[l] O \ar[r] & A
}
$$
Equivalently, for each node $x\in N'$, we have bijections $I'_x \simeq 
I_{\varphi x}$ and $O'_x \simeq 
O_{\varphi x}$, the subscripts indicating fibres.
Hence etale means arity preserving.
Denote by $\GR_{\et}$ the category of graphs and etale maps.
(The notion of etale map has a clear intuitive content.  It also fits
into the axiomatic notion of classes of etale maps
of Joyal-Moerdijk~\cite{Joyal-Moerdijk:openmaps}.)

A {\em graph inclusion} is a morphism which is levelwise 
injective.  A {\em subgraph} of a graph $G$ is an equivalence class of 
graph inclusions into $G$. 
An {\em open subgraph} is a subgraph whose inclusion is etale.

\end{blanko}

\begin{eks}
  The unique map $C^1_1 \to W_1$,
  $$
  \xymatrix{
  \{a,b\} \ar[d]& \ar[l] \{a\}\ar[d] \ar[r] & 1\ar[d] & \ar[l]\{b\}\ar[d] 
  \ar[r] & \{a,b\}\ar[d] \\
  1 & \ar[l] 1 \ar[r] & 1 & \ar[l] 1 \ar[r] & 1
  }
  $$
  is etale but not a graph inclusion.
\end{eks}

\begin{blanko}{Port-preserving maps.}
The inclusion of an edge is a diagram
$$
\xymatrix{
1 \ar[d]_e& \ar[l] 0\ar[d] \ar[r] & 0\ar[d] & \ar[l]0\ar[d] 
\ar[r] & 1\ar[d]^e \\
A & \ar[l] I \ar[r] & N & \ar[l] O \ar[r] & A
}
$$
It is an import precisely when the right-most square is a pullback, and it is an 
export precisely when the left-most square is a pullback.
Correspondingly, a morphism of graphs 
is called {\em import preserving} if the right-most square is a
pullback, and {\em export preserving} if the left-most square is a
pullback.
\end{blanko}

\begin{blanko}{Isolated nodes.}
  A node $x$ in a graph is called {\em isolated} when both $I_x$
  and $O_x$ are empty.
\end{blanko}
\begin{prop}\label{GGAA}
  Except in the case where $G'$ has an isolated node, a morphism of graphs
  $G'\to G$ is completely determined by its values on edges.  Precisely, 
  the natural map \linebreak $\Hom_{\GR}(G',G) \to \Hom_{\Set}(A',A)$ is injective.
\end{prop}

\begin{proof}
  The injectivity axiom implies that $A'\to A$ determines also $I'\to I$
  and $O' \to O$.  If $x \in N'$ is a node in $G'$, since it is assumed not
  to be isolated, it must be the image of some flag, either in $I'$ or in $O'$.
  In either case its image is forced by the image of that flag.
\end{proof}

\begin{blanko}{Relation with general graphs in the sense of Joyal-Kock.}
  Recall that according to \cite{Joyal-Kock:0908.2675}
  a {\em Feynman graph} is a diagram of finite sets
$$
\xymatrix {
E \ar@(lu,ld)[]_i & H \ar[l]_s \ar[r]^t & V
}
$$
such that $s$ is injective and $i$ is a fixpoint-free involution.

  The data of a directed graph in the sense of \ref{def:graph} can 
  equivalently be presented as 
  $$
  \xymatrix  @!0 @C=11pt  {
  \ar@(lu,ld)[]_-i \phantom{E}&A+A&&&&&&& \ar[lllllll]_-{t+s} I+O \ar[rrrrrrr]^-{\langle p,q 
  \rangle}&&&&&&&N,}
  $$
  where $i$ is the natural involution on $A+A$.
  Hence a directed graph has an
  underlying Feynman graph.  This is easily seen to be the object part of a 
  faithful functor from directed graphs (and etale maps) to Feynman graphs.
  In fact, directed graphs in the sense of \ref{def:graph}
  are precisely $D$-graphs for a certain graphical 
  species $D$, in the sense of~\cite{Joyal-Kock:0908.2675}.
  
  For further comparison between different implementations of the notion of 
  graph, see \cite{Batanin-Berger:1305.0086}.
\end{blanko}

\subsection{Connectedness and acyclicity}

We shall be concerned mostly with connected acyclic graphs.
We proceed to introduce these notions.

\begin{blanko}{Sums.}
  The category $\GR$ (as well as the subcategory $\GR_\et$) has categorical
  sums, and the empty graph
  $$
  \xymatrix{
  0 & \ar[l] 0 \ar[r] & 0 & \ar[l] 0 \ar[r] & 0 .
  }
  $$
  as initial object.  Sums are calculated level-wise.
  They amount to disjoint union of graphs.
\end{blanko}

\begin{blanko}{Connectedness.}
  Recall that $W_1$ is the graph $1=1=1=1=1$ with one node and one loop (in fact the 
  terminal object in $\GR$ (but not in $\GR_\et$)).
  A graph $X$ is {\em connected} if
  $$
  \Hom(X, W_1\!+\!W_1) = 2.
  $$
  In other words, $X$ is non-empty and every morphism to $W_1\!+\!W_1$ is 
  constant.  Equivalently, $X$  is  non-empty and is not a 
  sum of smaller graphs.  (Equivalently, in its most categorical formulation,
  a graph $X$ is connected when $\Hom(X, -)$ preserves finite sums.)
\end{blanko}

\begin{blanko}{Example.}
  A graph for which all the structure maps are bijections is 
  precisely a disjoint union of wheels.  In fact the full
  subcategory spanned by graphs of this type is
  equivalent to the category of finite-sets-with-a-permutation (by
  cycle-decomposition of permutations).
\end{blanko}

\begin{blanko}{Acyclicity.}
  A graph $X$ is called {\em acyclic} (or {\em wheel-free}) if
  $$
  \Hom(W_n, X) = 0, \quad\forall n>0 .
  $$
  In other words, $X$ does not admit a morphism from any 
  wheel, or equivalently, does not contain a wheel as a subgraph.
\end{blanko}

\begin{blanko}{Trees and linear graphs.}\label{trees}
    An acyclic graph is a {\em forest} when $q$ is a bijection.
    An acyclic graph is a {\em tree} when $q$ is a bijection and there is a 
    unique export (compare~\cite{Kock:0807.2874}).  It is a 
    {\em linear tree} (or {\em linear graph})
    if both $p$ and $q$ are bijections, and there is a 
    unique export.  Denote by $L_k$ the linear graph with $k$ nodes.
\end{blanko}
 
\begin{blanko}{Loops.}\label{loops}
  A {\em loop} is an edge which is simultaneously an input
  and an output for the same node.  In other words, $a\in A$ is a loop
  if there is a node $x$ such that $a \in I_x \cap O_x$.
  Equivalently, a loop in $X$ is the image edge of a map $W_1 \to X$.
  Accordingly, a node is {\em loopfree} if $I_x + O_x \to A$ is
  injective.  
  A graph is {\em loopfree}
  if every node is loopfree. 
  From the $W_1$-characterisation of loops, it is clear that
  an acyclic graph is loopfree.
\end{blanko}

\subsection{Closed-graph adjunction}

\begin{blanko}{Closed graphs.}
  A {\em closed graph} is a directed graph in the classical sense,
  i.e.~a presheaf on $\bullet \leftleftarrows \bullet$.  We use the
  standard letters
  $$
  \xymatrix{ E \ar[r]<+3pt>^s \ar[r]<-3pt>_t & V} .
  $$
  To a closed graph is associated a graph in the sense of \ref{def:graph}, namely
  $$
  \xymatrix{ E & \ar[l]_{=} E \ar[r]^t & V & \ar[l]_s E \ar[r]^{=} & E .}
  $$
  This defines a fully faithful functor from closed graphs to graphs.
  Its essential image is the subcategory of graphs for which the end maps
  are bijections.  We also call these closed graphs.
\end{blanko}

\begin{blanko}{The core of a graph.}
  The {\em core} of a graph $X$ is the closed graph 
  $O\times_A I \rightrightarrows N$ 
  given by the diagram
  $$
\xymatrix@!@-30pt{&& \ar[ld] O \times_A I \ar[rd] && \\
& \ar[ld]_q O \ar[rd]^t && \ar[ld]_s I \ar[rd]^p & \\
N & & A & & N  .
}
$$
It is denoted $X^\bullet$.  Viewed as a graph, the core is represented as 
follows, together with its canonical map to $X$:
$$
\xymatrix{
X^\bullet \ar[d]_\epsilon&: & O \times_A I\ar[d]& \ar[l]_{=} O \times_A I\ar[d] \ar[r] & N\ar[d] & 
\ar[l] O \times_A I\ar[d] 
\ar[r]^{=} & O \times_A I\ar[d] \\
X &:&A & \ar[l] I \ar[r] & N & \ar[l] O \ar[r] & A   .
}
$$
Taking core amounts to
deleting all ports (i.e.~replacing the set $A$  by
the subset $O \times_A I$ of inner edges), hence $X^\bullet$ is 
a subgraph of $X$.
It is clear that taking core is functorial, and it follows easily from the
universal property of the pullback that
\begin{prop}
  Taking core is right adjoint to the inclusion of closed graphs into 
  graphs.  The counit is $\epsilon$.
\end{prop}

%
%
\end{blanko}

\begin{blanko}{Core equivalences.}
  A graph map $f: Y \to X$ is called a {\em core equivalence} if $f^\bullet :
  Y^\bullet \to X^\bullet$ is invertible.
\end{blanko}

\begin{lemma}
  The etale maps are precisely the maps right orthogonal to the class of core
  equivalences between connected graphs, or equivalently, right orthogonal to
  the class of all maps between corollas.  The inclusions are precisely the maps
  right orthogonal to both $U+U\to U$ and $C_0^0+C^0_0 \to C^0_0$.
\end{lemma}

\subsection{Canonical neighbourhoods, covers and hulls}

\begin{blanko}{Canonical cover.}\label{cancov}
  Given a subset of nodes $N'\subset N$ of a graph $X$, construct a new graph
  as
$$
\xymatrix{
&I'+O' \ar[d]& \ar[l] I'\ar[d] \drpullback\ar[r] & N'\ar[d] & 
\ar[l]\dlpullback O'\ar[d] 
\ar[r] & I'+O'\ar[d] \\
X:&A & \ar[l] I \ar[r] & N & \ar[l] O \ar[r] & A   ,
}
$$
clearly a disjoint union of corollas.
Jointly they cover the nodes in $N'$.  When $N'=N$, this is
called the {\em canonical etale cover} of $X$, denoted $\cancov(X)$.
(It is an open cover iff $X$ is loopfree.)  When $N'$ consists of a single node 
$x\in N$, the construction gives the {\em canonical neighbourhood} of $x$,
an open subgraph when $x$ is loopfree.
\end{blanko}

\begin{blanko}{Open hull.}
  In the same situation, the {\em open hull} of $N'$ is defined by gluing the
  edges of the corollas according to their incidences in $X$, to obtain an open
  subgraph in $X$.  
  The notion of gluing will be formalised below.  In the present situation,
  it means taking union inside $A$ instead of disjoint
  union: simply take the image factorisation of $I'+O'\to A$.  Note that this
  includes also any existing loops at the nodes.  By the universal property of
  union, it is the smallest open subgraph having $N'$ as set of nodes.
\end{blanko}

\begin{blanko}{Etale hull.}\label{etalehull}
  Slightly more involved is the construction of the {\em etale hull of a
  subgraph}.  In this situation we are given a subgraph $G \subset X$, and we
  want to factor the inclusion as a core equivalence followed by an
  etale map.  The construction of flags is as before (forced by the etale
  requirement).  It remains to construct the correct edge set: it is a certain
  pushout, over the set of inner edges of $G$. 
  It will be important to consider a slightly more general situation.
A map of graphs $X'\to X$ is called {\em locally injective} when for each $x\in 
N'$ we have that $I'_x \to I_{\varphi x}$ and $O'_x \to O_{\varphi x}$ are injective.
\end{blanko}

\begin{prop}\label{prop:coreeq-etale}
    Given a locally injective map of graphs $f:X' \to X$, there exists a
    factorisation
    $$\xymatrix{
    X' \ar[rr]^f \ar[rd]_c && X \\
    & Y \ar[ru]_e &
    }$$
  where $c$ is a core equivalence and $e$ is etale.
  Among these factorisations, there is an essentially unique one for which
  $c$ is furthermore bijective on unit components.
\end{prop}
  
\begin{proof}
  Given
  $$
\xymatrix{
\phantom{::}X': \ar[d] & & A' \ar[d]& \ar@{ >->}[l] I'\ar[d] \ar[r]
& N'\ar[d] & 
\ar[l] O'\ar[d] 
\ar@{ >->}[r] & A'\ar[d] \\
\phantom{:}X: &&A & \ar@{ >->}[l] I \ar[r] & N & \ar[l] O \ar@{ >->}[r] & A
}
$$
we first take the following pullbacks:
  $$
\xymatrix{
A' \ar[dd]& \ar@{ >->}[l] I'\ar@{ >->}[d] \ar[r] & N'\ar@{=}[d] & 
\ar[l] O'\ar@{ >->}[d] 
\ar@{ >->}[r] & A'\ar[dd] \\
&  I''\ar[d] \drpullback\ar[r] & N''\ar[d] & 
\ar[l]\dlpullback O''\ar[d] 
 &  \\
A & \ar@{ >->}[l] I \ar[r] & N & \ar[l] O \ar@{ >->}[r] & A
}
$$
(the vertical injections by local injectivity).
These choices of $I''$, $N''$ and $O''$ are forced by the requirement that
the first map be bijcetive on nodes and the second etale.  It remains to
see if we can construct the edge set $A''$.  
The set $E''$ of {\em inner} edges of $Y$ must be $E'= O'\times_{A'} I'$, the set of 
inner edges of $X'$.
The minimal choice of $A''$ to achieve this is the pushout
$$
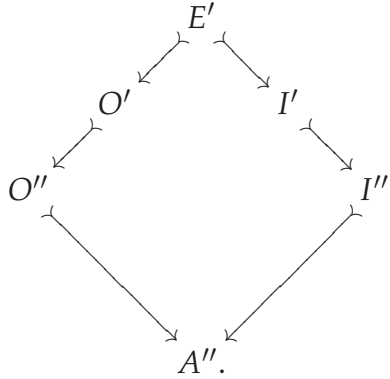
\xymatrix @=8pt @! {
  && \ar@{ >->}[ld] E' \ar@{ >->}[rd]&& \\
  & O' \ar@{ >->}[ld] && I'\ar@{ >->}[rd] & \\
 O''\ar@{ >->}[rrdd] &&&& I'' \ar@{ >->}[lldd] \\
 &&&&\\
&& A''  . &&
}
$$
(This is also a pullback, since the maps are injective.)
This choice amounts to giving $Y$ no isolated edges.  This works if also $A'$
is given as a pushout over $E'$ (which amounts to saying that already $X'$ had no
isolated edges):  in this case the map $A'\to A$ factors uniquely through $A''$
by the universal property of the pushout.  Otherwise, if $X'$ has isolated edges,
we need to add the same number of isolated edges to $Y$, that is, to add the 
same number of elements to $A''$.  Hence the extra requirement that the map 
$c$ should be bijective on unit edges hence fixes the choice of $A''$ uniquely.
\end{proof}
 

%
%
%
%
%

\subsection{Pushouts, coequalisers, and colimits over graphs}

\begin{blanko}{Gluing data.}
  A {\em shrub} is a disjoint union of unit graphs, i.e.~a graph of the form
  $$
    \xymatrix{
    S & \ar[l] 0 \ar[r] & 0 & \ar[l] 0 \ar[r] & S ,
    }
  $$
  where $S$ is a finite set.
  A {\em gluing datum} of a graph $G$ consist of a shrub $S$ with two injections 
  into $G$, 
    $$
    \xymatrix{ S \ar@{ >->}[r]<+3.5pt>^{\operatorname{ex}} \ar@{ 
    >->}[r]<-3.5pt>_{\operatorname{im}} & G ,}
    $$
  one export preserving, the other import preserving.
\end{blanko}

\begin{prop}\label{prop:et+bijnodes}
  The category of graphs $\GR$ admits coequalisers of gluing data
  $S \rightrightarrows G$.
  The quotient map $G \to Q$ is etale and bijective on nodes.
\end{prop}
\begin{proof}
  The coequaliser exists in the category of diagrams of shape \eqref{eq:graph}
  (i.e~without imposing the injectivity condition).  We just need to check the
  injectivity condition, i.e.~that $I+I' \to A+_S A$ is injective (and
  similarly for $O$).  But this is clear: to say that $e\in S$ maps to an export
  in $A'$ means it is not in the image of $s'$.  Hence the collapse $A+A' \to
  A+_S A'$ does not interfere with the injectivity of $I+I'\to A+A'$.
  Since the levelwise construction is just disjoint union at the level of
  nodes and flags, it is clear that the quotient map is bijective on nodes and 
  etale.
\end{proof}
\begin{cor}
  The category $\GR_{\et}$ of graphs and etale maps admits coequalisers of gluing data.
\end{cor}



The colimit of a gluing datum $S\rightrightarrows G$
is constructed by connecting exports to imports in $G$, realising one connection
for each unit graph in $S$.  Although it is a trivial observation, it will be
important to note that this colimit can be computed in steps by realising the
connections one by one in any order.

\begin{blanko}{Elementary graphs.}
    An {\em elementary graph} is a connected graph without inner 
    edges.  Up to isomorphism there are only the following: the 
    unit graph (one edge, no nodes), and the $(m,n)$-corollas.
    Let $\elGr\subset \Gr_\et$ denote the full subcategory spanned
    by the elementary graphs (and all etale maps).
    Hence the only maps are the inclusions
    of an edge into a corolla, and the permutations of imports or
    exports.
\end{blanko}

\begin{blanko}{Elements.}\label{elements}
  Let $X$ be a graph
  $
    \xymatrix{
    A & \ar[l]_s I \ar[r]^p & N & \ar[l]_q O \ar[r]^t & A .
    }
  $
  The comma category $\el(X) := \elGr\comma X$ 
  is called the {\em category of elements} of $X$.
  (See \cite{MacLane:categories}, Ch.II, \S 6, for the notion of comma 
  category.)
  It has the following explicit description. Its object
  set is $A+N$.  Its set of non-identity arrows is $I+O$.  An element $f\in I$
  has domain $s(f)$ and codomain $p(f)$; an element $g\in O$ has domain $t(g)$
  and codomain $q(g)$.  Since every arrow goes from an object in $A$ to an
  object in $N$, the category is really just a bipartite
  graph (with identity arrows added), namely the barycentric subdivision of $X$.
  
  The category of elements $\el(X)$ indexes canonically a diagram of 
  elementary graphs, $\el(X) \to \GR$, and $X$ is the colimit
  of this diagram in $\GR$:
  \begin{equation}\label{colim}
    X = \colim_{E\in \el(X)} E = \colim \big( \el(X) \to \GR\big).
  \end{equation}
  
  Assuming that $X$ has no unit components, the outer edges play no role in the
  colimit: it can be computed equally well indexed over $\el(X^\bullet)\subset 
  \el(X)$
  (although of course the ingredient elementary graphs are those of $X$, not
  those of $X^\bullet$).  The colimit indexed over $\el(X^\bullet)$ is that of
  a gluing datum: it is the coequaliser of the canonical cover 
  (the disjoint union of all
  the nodes in $X$ considered as corollas, cf.~\ref{cancov}) over the inner 
  edges:
  $$
  O \times_A I \rightrightarrows \cancov(X) \to X .
  $$ 
  More generally, we have:
\end{blanko}
\begin{lemma}
  For $X$ a graph without unit components, the functor $\el(X^\bullet) \to
  \el(X)$ is final.
\end{lemma}
\begin{cor}
  If $X \to Y$ is a core equivalence between graphs without unit components, 
  then $\el(X) \to \el(Y)$ is final.
\end{cor}
See~\cite{MacLane:categories}, Ch.IX, \S 3, for the  notion of final functor.
What this amounts to is that if $f:\el(X) \to \mathscr C$ is a functor, 
and if $\el(X^\bullet) \to\el(X)  \to \mathscr C$ admits a 
colimit, then so does  $f$, and the two colimits agree.

\begin{blanko}{Gluing datum from a graph of graphs.}\label{graphofgraphs}
  The case of interest in the previous discussion is the following.  A functor
  $\el(X) \to \GR$ is called a {\em graph of graphs} if it sends all $A$-objects
  to unit graphs, and all $I$-maps to import-preserving maps and all $O$-maps to
  export-preserving maps.  In this case (unless $X$ is has a unit component) the
  restriction to $\el(X^\bullet)$ is a gluing datum, and hence has a colimit.
  Therefore $\el(X) \to \GR$ has a colimit.  This is just the formal expression
  of the idea that if a graph $X$ is decorated with graphs at the nodes, then
  the decorating graphs (one for each node) can be glued together as prescribed
  by the incidence relations in the indexing graph.  (In the literature, this
  situation (subject to a further compatibility condition, see \ref{rcgg}) 
  is often referred to as {\em graph substitution} (see for
  example~\cite{Batanin-Berger:1305.0086} or \cite{Yau-Johnson}): the colimit is
  interpreted as the result of substituting each decorating graph into the
  corresponding node of the indexing graph.)
\end{blanko}

\begin{blanko}{Colimit formula for etale hull.}\label{hull=colim}
  If $H \to G$ is a subgraph, then the etale hull (\ref{etalehull}) is
  given by
  $$
  \colim \big( \el(H) \to \el(G) \to \GR \big).
  $$
  This just says that the etale hull is obtained by gluing together corollas 
  from $G$ according to the shape of $H$.
\end{blanko}

\begin{blanko}{Residue.}
  Denote by $\Cor$ the full subcategory of $\Gr_{\et}$ consisting of the 
  corollas.  Clearly $\Cor$ is a groupoid.
  The {\em residue} of a connected graph $G$ is the corolla having the same 
  imports and 
  exports as $G$. This defines a functor $\res : \Gr_{\iso} \to \Cor$
  (functorial in isomorphisms, not in general maps).  Note that $\res(U) = 
  C^1_1$.  If $x$ is a node of a graph, we shall also write $\res(x)$ for
  the residue of the canonical neighbourhood of $x$.
\end{blanko}

\begin{blanko}{Indexing graph for a gluing datum.}
  A coequaliser of a gluing datum $S \rightrightarrows G$
  can be interpreted as a colimit of connected 
  graphs as follows.  Write $S$ and $G$ as sums of connected components
  $$
  \sum_i U_i \rightrightarrows \sum_k G_k .
  $$
  Form the graph $\sum R_k$ by replacing
  each $G_k$ with its residue $R_k=\res(G_k)$.
  We still have the maps from the $S$ edges into these corollas, and we can
  take the colimit $R$, which we call the {\em indexing graph}.
  The original diagram is now over $\el(R^\bullet)$, 
  the category of elements of $R^\bullet$.
\end{blanko}

\begin{lemma}\label{connQconnR}
  The coequaliser $Q$ is connected if and only if the indexing graph $R$ is
  connected.  Furthermore, in this situation, $\res(Q) \simeq \res (R)$.
\end{lemma}

\begin{proof}
  To give a map $Q \to W_1+W_1$ is the same as giving a cocone
  $S \rightrightarrows \sum G_k 
  \to W_1+W_1$.  Since each of the $G_k$ is connected, each map $G_k \to 
  W_1+W_1$ is constant, and hence amounts to giving a cocone
  $S \rightrightarrows \sum R_k \to W_1+W_1$, and hence
  a map $R \to W_1+W_1$.  Hence $Q$ 
  is connected if and only if $R$ is.  The second statement follows 
  easily since by construction $\sum G_k$ and $\sum R_k$ have the same set of imports,
  of which the same subset $S$ is spent with gluing, so that also $Q$ and $R$ 
  are left with the same set of imports.  Ditto with exports.
\end{proof}

\begin{lemma}
  If all the individual graphs $G_k$ are acyclic, and if the indexing
graph $R$ is acyclic then the coequaliser $Q$ is acyclic too.
\end{lemma}
    
\begin{proof}
  Let $W \to Q$ be a wheel in $Q$.  For each of the connected graphs $G_k$,
  consider the pullback
  $$\xymatrix{
     H_k \drpullback \ar[r]\ar[d] & G_k \ar[d] \\
     W \ar[r] & Q
  }$$
  Since we have assumed $G_k$ acyclic, each $H_k$ is a sum of linear graphs, 
  and each $H_k \to G_k$ preserves imports and exports.  This means
  that it makes sense to take residue of
  each $H_k \to G_k$, yielding in each case a map $T_k \to R_k = 
  \operatorname{res}(G_k)$,
  where each $T_k$ is a sum of $L_1$-graphs.
  These linear graphs
  $L_1$ glue together to give a wheel in $R$.  (It is closed because each
  import in one string corresponds to an export in another string.)
\end{proof}

\begin{lemma}
  If the indexing graph $R$ is loopfree, then each of the maps $G_k \to Q$ is an
  open inclusion.
\end{lemma}

\subsection{Complements and convexity}

The material in this subsection will only be needed again in 
Subsection~\ref{sec:wtilGr}.

\begin{blanko}{Naive complement.}
  If $H \to G$ is a subgraph (i.e~levelwise injective), then the {\em naive
  complement} is simply defined by taking complements levelwise.  It is again a
  subgraph.  This notion is not very useful for the present purposes, where the
  emphasis in on etale maps.  The better notion is the following
  adjustment.
\end{blanko}

\begin{blanko}{Etale complement.}
  If $H \to G$ is a subgraph, we define the {\em etale complement}
  to be the etale hull of the naive complement.  
  We denote it
  $\compl H$ or $G\shortsetminus H$.  Precisely, if temporarily $H'$
  denotes the naive complement, then $H' \to G$ is injective and in particular
  locally injective, and we take its core-equivalence/etale factorisation 
  (\ref{prop:coreeq-etale})
      $$\xymatrix{
    H' \ar[rr] \ar[rd]_{\text{core eq.}} && G   . \\
    & \compl H \ar[ru]_{\text{etale}} &
    }$$
    By the colimit formula for etale hull (\ref{hull=colim}), the complement can
    also be described as
    $$
    \compl H = \colim\big( \el(H') \to \el(G) \to \GR \big).
    $$
\end{blanko}
  
\begin{blanko}{Etale complement of open subgraphs.}
  In the special case where $H \to G$ is an connected open subgraph, 
  which contains at least one node,
  then $G\shortsetminus H$ is just the open hull of the remaining nodes in 
  $G$.  In particular, $\compl H$ is again an open subgraph of $G$ (but
  not in general connected).
  The intersection $S := H \cap \compl H$ is a set of edges, namely the ports 
  of $H$ that are not also ports of $G$, and $G$ can be recovered by gluing 
  together $H$ and $\compl H$ along $S$.  Precisely,  
  $G$ is naturally the colimit of 
  the gluing datum
  \begin{equation}\label{SHHG}
  S \rightrightarrows H + \compl H .
  \end{equation}
\end{blanko}
 
\begin{blanko}{Etale complement of an edge.}
  We shall also need the very special case where $H$ consists of a single edge 
  of $G$.
  If $H$ consists of a port of $G$, then $\compl H = G$.
  If $H$ consists of an inner edge $e$ of $G$, then
  $G \shortsetminus e$ is
  the graph obtained from $G$ by cutting that edge.    More formally, 
  just as $G$ is the colimit of the canonical gluing datum
  $$
  O \times_A I \rightrightarrows \cancov G
  $$
  (of all the nodes over all the inner edges), $G\shortsetminus e$ is 
  the colimit of
  $$
    O \times_A I \shortsetminus e \rightrightarrows \cancov G.
  $$
  Hence there is a natural
  map $G\shortsetminus e \to G$ which is bijective on nodes and etale but {\em not} an 
  inclusion: the inverse image of $e$ consists of two edges.
\end{blanko}

\begin{blanko}{Convex open subgraphs.}
  Recall (from \ref{trees}) that $L_k$ denotes the linear graph with $k\geq 0$ nodes.
  A {\em path} in a graph $G$ is a map $L_k \to G$ (not required etale).
  An connected open subgraph $H \subset G$ is called {\em convex} if the 
  inclusion map is right orthogonal to every inclusion of the ports into a linear graph:
  $$\xymatrix @!@-8pt {
   U+U \ar[r] \ar[d] & H \ar[d] \\
   L_k \ar[r] \ar@{..>}[ru]^{\exists!}& G .
  }$$
  Precisely, $U+U\to L_k$ is import-preserving on the first summand and 
  export-preserving on the second summand.  (The horizontal maps and the diagonal 
  filler are not required to be etale.)
\end{blanko}

\begin{lemma}\label{convcomp}
  The composite of two convex inclusions is again convex.  The identity map of 
  a connected graph is a convex inclusion.  Hence graphs and convex inclusions
  form a subcategory of $\Gr$.
\end{lemma}
\begin{proof}
  This is clear since convex inclusions form a right orthogonal class.
\end{proof}

\begin{blanko}{Edge poset.}
  If $X$ is a graph, then
  $$
  I \times_N O \to A \times A
  $$
  is a relation on $A$, denoted $ \lessdot $: we have $x \lessdot y$ iff
  there is a
  node with $x$ as incoming edge and $y$ as outgoing edge.  An edge $x$ is a
  loop iff $x \lessdot x$, so that $X$ is loopfree iff $\lessdot$ is
  anti-reflexive.  The graph is acyclic iff the transitive closure of $\lessdot$
  is anti-symmetric and anti-reflexive.
  
  Assuming that $X$ is acyclic, the transitive and
  reflexive closure of $\lessdot$, denoted $\leq$,
  is a poset called the {\em edge poset} of $X$.
  We have $x \leq y$ iff there is a path from $x$ to $y$  in the graph.
  Hence a connected open subgraph $H \subset X$ is convex iff the edge poset of
  $H$ is a convex subset of the edge poset of $X$, in the usual sense
  ($x,y\in H \;\&\; x\leq m \leq y \Rightarrow m \in H$).  
\end{blanko}

\begin{blanko}{Remark.}
  The notion of convexity makes sense of course also for subgraphs
  not required to be open or not required to be connected,
  but then a looser notion of path is needed to detect convexity:
  specifically, `linear' graphs starting or ending in a node instead of
  an edge should be allowed
  
  Conversely, it is also possible to formulate the notion of convexity
  entirely inside the category of etale maps.  In this case the notion
  of path must be replaced by the notion of strip: a {\em strip} is the etale
  hull of a path.  In the lifting diagram, instead of just linear graphs $L_k$
  it is necessary to use all the graphs $P$ such that $P^\bullet \simeq 
  L_k^{\bullet}$.  Indeed, these are precisely the graphs that can appear
  in the core-equivalence/etale factorisation of a path, as in 
  \ref{prop:coreeq-etale}.
%
\end{blanko}

\begin{blanko}{Convexity in terms of etale complements.}
  Let $H$ be a connected open subgraph of an acyclic graph $G$.  If $H$ is a
  single edge, clearly $H$ is convex.  So we assume now that $H$ contains at
  least one node.  The intersection  $S:=H \cap\compl H$
  is a disjoint union of edges.
  Convexity says that there are no paths starting and ending in $H$ that go
  through a node of $\compl H$.  Such an offending path must necessarily go
  through an edge in $S$ (we must leave $H$ somewhere), then through at least
  one node in $\compl H$, and later through another edge in $S$ (we must enter $H$
  again somewhere).  So equivalently we can say that there is no path in $\compl
  H$ between two distinct edges of $S$.
  In particular, we see that convexity of $H$ does not really depend on what is
  inside $H$, but only on its boundary and its relationship with the complement.  
\end{blanko}

\begin{blanko}{Terminology from Hackney-Robertson-Yau~\cite{Hackney-Robertson-Yau}.}
  Two important notions in \cite{Hackney-Robertson-Yau} can be expressed in
  terms of convexity: Two nodes in a connected graph are {\em closest
  neighbours} \cite{Yau-Johnson}, when the open hull of the two nodes 
  is (connected and) convex.  A
  node of a connected graph is {\em almost isolated} if the complement is
  (connected and) convex, or if it is the only node in the graph.  (The reader
  is referred to \cite{Hackney-Robertson-Yau} to see how these notions are
  defined there (without the notion of convexity), and how they are exploited,
  among other things, to arrive at the notion of subgraph used there (which in
  the present terminology is the notion of convex open subgraph).)
\end{blanko}

\section{Properads}

\label{sec:properads}

This section owes a lot to Joyal-Kock~\cite{Joyal-Kock:0908.2675}.

\bigskip

Denote by $\Gr$ the category of connected acyclic graphs
with etale maps.  From now on we say simply {\em graph}
for connected acyclic graph.

\subsection{Digraphical species (coloured bi-collections) and $F$-graphs}

\begin{blanko}{Digraphical species = coloured bi-collections.}
  A presheaf 
  \begin{eqnarray*}
    F:\elGr\op & \longrightarrow & \Set  \\
    E & \longmapsto & F[E]
  \end{eqnarray*}
  is called a {\em digraphical species} or a {\em (coloured) bi-collection}.
  So it is the
  data of a 'set of colours' $F[U]$, the value on the unit graph, and for each $(m,n)$, a
  set $F[m,n]$ of operations of biarity $(m,n)$.  Each input and output slot in
  an operation is labelled by a colour, via the maps in $\elGr$ from the unit
  graph to the corollas.
  
  We favour the terminology {\em digraphical species} when $F$ conveys the idea
  of a structure on digraphs, something to decorate digraphs with, while we 
  prefer the terminology {\em bi-collection} when $F$ serves as the structure
  underlying (or freely generating) a properad.
    
    A graph $X$ defines a bi-collection  by
    \begin{eqnarray*}
      X: \elGr\op & \longrightarrow & \Set  \\
      E & \longmapsto & \Hom^\et(E,X).
    \end{eqnarray*}
    The category of elements $\el(X)$ introduced in \ref{elements} then 
    coincides
    with the category of elements of this presheaf (see 
    \cite{MacLane:categories}, Ch.III, \S 7).
\end{blanko}

\begin{blanko}{Grothendieck topology and sheaves.}
  The category $\Gr$ has a natural Grothendieck topology in which a cover is a
  collection of etale maps that are jointly surjective on nodes and on edges.
  Every graph has a canonical cover $\cancov(X) \to X$ which is a disjoint union
  of elementary graphs, cf.~\ref{cancov}.
  From \eqref{colim} we get
  $$
  \PrSh(\elGr) \simeq \Sh(\Gr) .
  $$

Via this equivalence, a digraphical species $F$
can be evaluated on any graph, not just on the elementary
ones.  The formula is:
$$
F[G] = \lim_{E\in \el(G)} F[E] .
$$
(We already know that $G$ is a colimit of its canonical diagram,
in the category $\Gr$ of graphs and etale maps.
That $F$ is a sheaf on $\Gr$ means precisely that this colimit is 
sent to a limit, which is the one in the formula.)
\end{blanko}

\begin{lemma}\label{lem:F[G]}
  In the special case where the presheaf $F: \elGr\op\to\Set$ is `represented' 
  by a graph $X$, that is, $F[E] = \Hom^\et(E,X)$, then as a sheaf 
  $F:\Gr\op\to\Set$ it is genuinely represented by $X$:
$$
F[G] = \Hom^\et(G,X) .
$$
\end{lemma}
\begin{proof}
$$ 
F[G] = \!\lim_{E\in \el(G)} F[E] = \!\lim_{E\in \el(G)} \Hom^\et(E,X) = 
\Hom^\et (\colim E, X) =
\Hom^\et(G,X) .\vspace{-12pt} 
$$  
\end{proof}

\begin{blanko}{$F$-graphs.}
  Every digraphical species $F$ defines a notion of {\em $F$-graph}.  They are
  graphs whose edges are decorated by the colours of $F$, and whose nodes are
  decorated by the operations of $F$, subject to obvious compatibility
  conditions.  Formally, the category of $F$-graphs is the comma category
  $\Gr\comma F$ (or $\GR \comma F$ if we allow non-connected graphs).  If for a
  moment we denote by $T$ the terminal digraphical species, then $T$-graphs are
  the same thing as graphs, in the following discussion called {\em naked
  graphs}.
  
  All the basic results about graphs hold also for $F$-graphs.  Coequalisers of
  gluing data exist for $F$-graphs, just as for graphs, now over $F$-shrubs,
  but the indexing graph is a naked graph, not an $F$-graph: the residue of an
  $F$-graph is not naturally an $F$-graph, only a naked graph.  (For example,
  $F$ could be the $(2,1)$-species, whose $F$-graphs are binary trees.  The
  residue of a binary tree is of course not in general a binary tree.)
  Similarly, a graph of $F$-graphs $\el(R) \to \GR\comma F$ admits a colimit
  in $\GR\comma F$ (obtained by gluing together all the $F$-graphs according
  to the incidence relations expressed by the naked graph $R$).
  
  (Again we see that the indexing graph is not on the same footing as the
  ingredients of the colimit.  There is no natural notion of substituting an
  $F$-graph into the node of a naked graph.  What does make sense is to use the
  naked graph as a shape of colimit to compute in the category $\GR\comma F$.
  It's a recipe for gluing.  Hence in the present formalism, gluing is
  the fundamental notion, while
  substitution is derived from it.)
\end{blanko}

\subsection{The free-properad monad}

\label{sec:freeproperad}

We shall define properads as algebras for a certain monad.  This properad monad
was also described by Hackney-Robertson-Yau~\cite{Hackney-Robertson-Yau} (Ch.2),
and in a more general setting by Yau-Johnson~\cite{Yau-Johnson} (Ch.10--11), in
both cases in terms of graph substitution. 
The present description of the monad
is literally the same as that for coloured modular operads of
Joyal-Kock~\cite{Joyal-Kock:0908.2675}, where in turn it is mentioned that it is
just the coloured version of the construction of
Getzler-Kapranov~\cite{Getzler-Kapranov:9408}.

\begin{blanko}{$(m,n)$-graphs.}
  An {\em $(m,n)$-graph} is a graph $G$
  equipped with an isomorphism $\res(G) \simeq C^m_n$.  More formally, the 
  groupoid of $(m,n)$-graphs $(m,n)\kat{-Gr}_{\iso}$ is the homotopy fibre 
  over $C^m_n$ of the functor $\res: \Gr_{\iso} \to \Cor$.
Note that $\operatorname{res} : \Gr_{\iso} \to \kat{Cor}$ is a groupoid
fibration, not a discrete fibration, since a graph may well have automorphisms
that fix all ports.
We are interested in its fibrewise $\pi_0$: that's the essential part of a
digraphical species, which to $(m,n)$ assigns
$\pi_0((m,n)\kat{-Gr}_{\iso})$.  (It does not say anything about colours,
the values on the unit graph, but this information will be provided
automatically in the construction below.)

If $F$ is a digraphical species,
there is a residue functor $\res:\Gr_{\iso}\comma F \to \Cor$ from $F$-graphs 
to naked corollas.  The groupoid of $(m,n)$-$F$-graphs,
denoted $(m,n)\kat{-Gr}_{\iso} \comma F$, is the homotopy fibre over $C^m_n$
of this functor.
\end{blanko}

\begin{blanko}{Underlying endofunctor of the free-properad monad.}
  Let $\triv$ denote the unit graph.  We define the {\em monad for properads}:
  \begin{align*}
   \PrSh(\elGr) \ \longrightarrow & \ \PrSh(\elGr) \\  
   F \ \longmapsto & \ \ov F ,
  \end{align*}
  where $\ov F$ is the bi-collection given by $\ov F[\triv] :=  
  F[\triv]$ and
  \begin{align*}
  \ov F[m,n] \ := & \ \colim_{G \in (m,n)\kat{-Gr}_{\iso}} F[G] \\[6pt]
  =& \ \sum_{G\in \pi_0((m,n)\kat{-Gr}_{\iso})}  \frac{F[G]}{\Aut_{(m,n)}(G)} \\[6pt]
  =& \ \pi_0 \big((m,n)\kat{-Gr}_{\iso} \comma F \big) .
  \end{align*}
  Here the first equation follows since $(m,n)\kat{-Gr}_{\iso}$ is just a 
  groupoid: the sum is over isomorphism classes of $(m,n)$-graphs, and 
  $\Aut_{(m,n)}(G)$ denotes
  the automorphism group of $G$ in $(m,n)\kat{-Gr}_{\iso}$.
\end{blanko}

\begin{blanko}{Multiplication for the monad.}
  $\ov F[m,n]$ is the set of isomorphism classes of $(m,n)$-$F$-graphs: it is
  the set of ways to decorate $(m,n)$-graphs by the digraphical species $F$.
  Now $\ov { \ov F} [m,n]$ is the set of $(m,n)$-graphs decorated by $F$-graphs:
  this means that each node is decorated by an $F$-graph with matching ports.
  We can use the $(m,n)$-graph as indexing a diagram of $F$-graphs, and then
  take the colimit.  This describes the monad multiplication
  $$
  \mu_F: \ov{\ov F} \to \ov F .
  $$

  More formally, the groupoid $\Gr_{\iso}\comma \ov F$ has as objects pairs
  $(R,\phi)$ where $R$ is a graph, and $\phi: \Hom( -, R) \to \ov F$ is a
  natural transformation.  Equivalently we can regard $\phi$ as a functor
  $\el(R) \to \elGr\comma \ov F$.  Now there is also a canonical functor
  $\elGr\comma \ov F \to \Gr\comma F$, which takes unit graphs to unit graphs,
  and takes a corolla decorated by an $F$-graph to that same $F$-graph.
  The composite functor
$$
\el(R) \to  \elGr\comma \ov F \to \Gr\comma F
$$
is a graph of $F$-graphs in the technical sense of \ref{graphofgraphs},
and we take its colimit to obtain a single $F$-graph.
The whole construction defines a functor
$$
\Gr_{\iso}\comma \ov F \to \Gr_{\iso}\comma F ,
$$
which is compatible with taking residue by Lemma~\ref{connQconnR}.
The fibrewise $\pi_0$ of this functor defines the monad multiplication.
\end{blanko}

\begin{blanko}{Associativity.}
  Associativity asserts that this square commutes:
  $$\xymatrix{
     \ov{\ov{\ov F}}\ar[r]^{\ov {\mu_F}}\ar[d]_{\mu_{\ov F}} & \ov{\ov F} 
     \ar[d]^{\mu_F} \\
     \ov {\ov F} \ar[r]_{\mu_F} & \ov F .
  }$$
  The elements in $\ov {\ov {\ov F}}[m,n]$ are graphs of graphs of $F$-graphs,
  and associativity amounts to saying that these $F$-graphs can be glued
  together in two ways with the same result.  In detail, an element in $\ov {\ov
  {\ov F}}[m,n]$ is a graph of $\ov F$-graphs, so it amounts to a
  graph $R$, and for each node $x$ in $R$ an $\ov F$-graph $A_x$ (and for each
  inner edge in $R$ the corresponding ports of the $A_x$ match).  Each $\ov
  F$-graph $A_x$ is a graph of $F$-graphs, so for each node in each $A_x$ there
  is an $F$-graph (and again compatibilities).  So altogether there is a number
  of $F$-graphs involved;  associativity says that the following two ways of
  gluing them all together give the same result.  Either {\em inner-first}
  (that's $\mu_F \circ \ov{\mu_F}$): we first glue together, for each $A_x$
  separately, the corresponding $F$-graphs, to obtain a set of bigger $F$-graphs
  indexed by the nodes in $R$, and then finally glue together these bigger
  $F$-graphs according to the colimit shaped by $R$.  Or {\em outer-first}
  (that's $\mu_F \circ\mu_{\ov F}$): we first prepare the overall shape by
  gluing together all the graphs $A_x$ according to the shape $R$.  This
  produces a graph $Q$ of $F$-graphs, and then we use $Q$ as recipe for gluing
  all the $F$-graphs.  Note that there is a natural bijection between the nodes
  in $Q$ and the sum of all the nodes in all the $A_x$ --- this follows because
  the quotient map of the gluing construction $S \rightrightarrows \sum_x A_x
  \to Q$ is bijective on nodes (\ref{prop:et+bijnodes}).  To see that the two gluing constructions agree,
  assume first that none of the $A_x$ are unit graphs.  Start with the
  outer-first gluing: here we are simply gluing all the $F$-graphs according to
  one graph $Q$.  However, this graph $Q$ contains as open subgraphs all the $A_x$.
  We can perform the colimit construction by first gluing separately over the
  inner edges of each of the $A_x$ (in each case this is a subset of the inner
  edges in $Q$, and all these subsets are disjoint).  But this first step is precisely to assemble all the
  $F$-graphs according to which $A_x$ they belong to, so it is precisely the
  first step in the inner-first gluing prescription.  Finally we glue along the
  remaining inner edges in $Q$.  By the assumption that none of the $A_x$ are
  unit graphs, these remaining inner edges are precisely identified with the
  inner edges of the outermost graph $R$.  So under this assumption, both ways
  of gluing are over the same sets of edges.  Finally we can easily reduce to
  this situation from the general case: if there is a node $x$ in $R$ such that
  $A_x$ is a unit graph, then we can start the colimit computation (in either
  way) by taking the pushout over any inner edge incident to $x$.  This pushout
  does not affect the result, neither the graph $Q$ in the outer-first
  calculation, nor the gluing of the $\ov F$-graphs $A_x$ in the inner-first
  calculation.  We may therefore as well assume that there are no nodes of this
  type in $R$.
\end{blanko}

\begin{blanko}{Unit for the monad.}
  The unit for the monad is given by interpreting an $F$-corolla $C\in F[m,n]$
  as an $F$-graph.  The unit law says: (1) given an $(m,n)$-$F$-graph $X$,
  interpreting it first as an $(m,n)$-corolla of $F$-graphs (the single
  $F$-graph $X$ itself), and then taking the (trivial) colimit, that gives back
  the $F$-graph $X$ again; and (2), interpreting $X$ as a graph of $F$-corollas,
  and then taking the colimit of these corollas, also gives back the original
  $F$-graph $X$.  Both cases are clear.
\end{blanko}

\begin{blanko}{Properads.}
  A {\em (coloured) properad} is defined to be an algebra for the properad monad
  $F\mapsto \ov F$.  This means that it is a bi-collection $F:\elGr\op\to\Set$
  equipped with a structure map $\ov F \to F$ obeying a few easy axioms
  (cf.~\cite{MacLane:categories}, Ch.~VI): it amounts to a rule which for any
  $(m,n)$-graph $G$ gives a map $F[G] \to F[m,n]$, i.e.~a way of constructing a
  single operation from a whole graph of them.  This rule satisfies some
  associativity conditions, amounting to independence of the different ways of
  breaking the computation into steps.  Let $\kat{Prpd}$ denote the category of
  algebras for the properad monad $F\mapsto \ov F$.
\end{blanko}

\begin{blanko}{Some variations.}
  Polycategories~\cite{Szabo:MR0373846}, also called
  dioperads~\cite{Gan:0201074}, are obtained by using only simply connected
  graphs (and the same elementary graphs).  Operads are obtained by using only
  rooted trees (and then only elementary graphs that are rooted trees),
  cf.~\ref{trees}.  See Kock~\cite{Kock:0807.2874}.  Categories are obtained by using
  only linear graphs (and elementary graphs that are linear).
\end{blanko}

\subsection{Generic/free factorisation and nerve theorem}
\label{sec:generic}

This subsection and the next, not really used elsewhere in the paper, introduce
and study a bigger category of graphs, whose new maps are generated by the
free-properad monad.  One important aspect of this bigger category $\wtil\Gr$ is
the nerve theorem (\ref{thm:nerve}), characterising properads among presheaves
on $\wtil\Gr$ in terms of a Segal condition.  The category $\wtil\Gr$ is also
described by Hackney-Robertson-Yau~\cite{Hackney-Robertson-Yau}, although they
are more interested in a smaller category (see \ref{hry-cat} below).


\begin{blanko}{Kleisli category.}
  We consider the diagram
  \begin{equation}\label{diagramforPhi}
    \xymatrixrowsep{48pt}
    \xymatrixcolsep{72pt}
    \xymatrix @!0
    {&\wtil\Gr \ar[r]^{\text{f.f.}} & \kat{Prpd} \ar@<1ex>[d]^{\text{forgetful}} \ar@{}[d]|{\dashv} \\
  \elGr \ar[r] &\Gr \ar[u]^{\text{i.o.}} \ar[r]_-{a} & \PrSh(\elGr) 
  \ar@<1ex>[u]^{\text{free}}
  }
  \end{equation}
  obtained by factoring $\Gr \to \kat{Prpd}$ as identity-on-objects followed by
  fully faithful.  In other words, $\wtil\Gr$ is the Kleisli category of the
  monad (see \cite{MacLane:categories}, Ch.VI, \S 5), restricted to $\Gr$. 
  This means that a morphism in $\wtil\Gr$ from graph $R$
  to graph $Y$ is defined as a morphism of bi-collections from $R$ to $\ov Y$.
  So where the original maps (those coming from $\Gr$, now called {\em free maps}) send
  vertices to vertices, the general maps in $\wtil\Gr$ send
  vertices to `subgraphs' --- more precisely, a vertex $x$ of $R$ is sent to an etale
  map $G_x\to Y$, in both cases subject to compatibility conditions.
  These conditions say that all the $G_x$ form a residue-compatible 
  graph of graphs (\ref{rcgg}) indexed by $R$, such that
  the colimit $Q$ comes with an etale map to $Y$. 
  In particular, in the bigger
  category $\wtil \Gr$ there is a new kind of map from $R$ to $Q$ which can be described
  as {\em refining} each of the nodes in $R$, as detailed below.
  This map realises the construction of
  `substituting the graphs $G_x$ into the nodes of $R$'. 
  The second step in the general
  map $R \to Y$ is the etale map $Q \to Y$.  Hence we see that
  every map in $\wtil \Gr$
  factors as a refinement followed by an etale map.
  This is an example of generic/free factorisation, an important general 
  phenomenon, and a key ingredient in achieving the nerve theorem below 
  (\ref{thm:nerve}).
\end{blanko}

\begin{blanko}{Residue-compatible graphs of graphs.}\label{rcgg}
  According to \ref{graphofgraphs}, a {\em graph of $F$-graphs} is a functor
  $\gamma:\el(R) \to \Gr\comma F$ that sends all $A$-objects to unit $F$-graphs,
  and all $I$-maps to import-preserving maps and all $O$-maps to
  export-preserving maps.  We say that a graph of $F$-graphs $\gamma$ is {\em
  residue compatible} when for each node $x$ in $R$, we have $\res(x) = \res(
  \gamma(x))$.
\end{blanko}

\begin{blanko}{Refinements.}
  We first treat the case where the domain is a corolla.
  To give a map in $\wtil \Gr$ from $C^m_n$ to a graph $Y$ is to give a map of 
  presheaves $C^m_n \to \ov Y$, i.e.~an element in $\ov Y[m,n]$.  By 
  construction
  this is given by (an isoclass of) a graph $G \in (m,n)\kat{-Gr}_{\iso}$
  together with an element $[\varphi] \in \Hom(G,Y)/\Aut_{(m,n)}(G)$.
  We call such a map a {\em refinement} if $\varphi$ is invertible. 
  It is now clear that we have the following factorisation into a refinement 
  followed by a free map (i.e.~the image of an etale map):
  $$\xymatrix{
     C^m_n \ar[rr]^{(G,[\varphi])}\ar[rd]_{(G,[\id])} && \ov Y . \\
     & \ov G \ar[ru]_{\ov \varphi} & 
  }$$
  This factorisation is not unique, since each $\sigma\in \Aut_{(m,n)}(G)$
  yields a different representative $\varphi \circ \sigma$ for the class
  $[\varphi]$.  But there is clearly a (free) isomorphism between such two
  factorisations, simply given by $\sigma$:
  $$\xymatrix{
  & \ov G \ar[rd]^{\ov{\varphi\sigma}}  \ar[dd]^{\ov\sigma}&\\
     C^m_n \ar[ru]^{(G,[\id])}\ar[rd]_{(G,[\id])} && \ov Y .\\
     & \ov G \ar[ru]_{\ov \varphi} & 
  }$$
  Note that the left-hand triangle commutes because $[\sigma] =[\id]$ modulo 
  $\Aut_{(m,n)}(G)$.
  The same diagram also shows that the factorisation, although it is unique up
  to isomorphism, is not in general unique up to {\em unique} isomorphism: if 
  $\sigma$ is a nontrivial port-preserving deck transformation of
  $\varphi$ (this can only happen when there are no ports), 
  then $\varphi= \varphi\sigma$, and the diagram represents a nontrivial
  automorphism of a factorisation.
  
  The general map in $\wtil\Gr$, say $R \to Y$, is essentially a colimit of maps
  of the previous form.  More formally, it is given by a residue-compatible
  graphs of $Y$-graphs (cf.~\ref{rcgg})
  $$
  \gamma:\el(R) \to \Gr\comma Y .
  $$
  The colimit of $\gamma$ is a graph 
  $Q$
  with an etale map to $Y$.  The map $R \to Y$ is a {\em refinement} if
  this etale map is invertible.  In the general case, $R \to Q \to Y$
  constitutes the refinement/free factorisation.
  In conclusion:
\end{blanko}
\begin{prop}\label{factor}
  Every map in $\wtil\Gr$ factors as a refinement followed by a free map.  This
  factorisation is unique up to (non-unique) free isomorphism.
\end{prop}
A version of this factorisation is also obtained in 
\cite{Hackney-Robertson-Yau} (Lemma~5.43).

\begin{BM}
  In the preceding discussion, $Y$ was assumed to be a graph, but it fact this
  is irrelevant: the arguments work exactly the same for $Y$ a general presheaf.
  In any case a map $R \to Y$ in the Kleisli category is given by $\gamma:
  \el(R) \to \Gr\comma Y$, subject to the same conditions as above, and in any
  case the middle object $Q$ appearing in the factorisation $R \to Q \to Y$ is a
  graph.  This will be important in the proof of the nerve theorem.
\end{BM}

\begin{blanko}{Generic maps (cf.~\cite{Weber:TAC13}).}
  The refinement/etale factorisation is an instance of a very general phenomenon,
  that of generic/free factorisations and monads with arities, introduced and
  studied in depth by Weber~\cite{Weber:TAC13}, \cite{Weber:TAC18}.  A
  recommended entry point to the theory is
  Berger-Melli\`es-Weber~\cite{Berger-Mellies-Weber:1101.3064}.

  Let $T : \CC\to\CC$ be a monad.  A {\em (weakly) generic map} is a map $g:A 
  \to TG$
  such that for every map $f:X\to Y$ in $\CC$ and every solid square
  $$\xymatrix @! {
     A \ar[r]^e\ar[d]_g & TX \ar[d]^{Tf} \\
     TG \ar[r]_{Th} \ar@{..>}[ru]& TY
  }$$
  there exists a diagonal filler $d:G\to X$ (i.e.~such that $f\circ d = h$ and
  $Td \circ g = e$).  A monad is said to {\em admit (weak) generic
  factorisations} if every map $A\to TY$ admits a factorisation as a (weakly)
  generic map followed by a free map.  (Note that we talk about generic maps in
  the weak sense of \cite{Weber:TAC13}, not in the strict sense of
  \cite{Weber:TAC18}.)
\end{blanko}

\begin{lemma}
  The refinement maps are (weakly) generic.
\end{lemma}

\begin{proof}
  Given a square as in the definition of generic, with $g$ a refinement map,
  factor the top map $e$ as refinement followed by free.  We now have two 
  different refinement/free factorisations, so by the previous proposition, there 
  exists a free isomorphism comparing them.  This provides a (free) filler in
  the square. 
\end{proof}

\begin{blanko}{Nerve functor.}
  The embedding $i:\wtil\Gr\to\kat{Prpd}$ induces the nerve functor
  \begin{align*}
    N: \kat{Prpd} \ &  \longrightarrow \ \PrSh(\wtil\Gr)\\
    X \ &\longmapsto \ \Hom_{\kat{Prpd}}(i(  \_ ) , X)
  \end{align*}
  featured in the nerve theorem:
\end{blanko}

\begin{thm}\label{thm:nerve}
  The nerve functor $N : \kat{Prpd} \to \PrSh(\wtil\Gr)$ is fully faithful,
  and a  presheaf is in the essential image of $N$ if and only if it satisfies 
  the Segal condition, i.e.~its restriction to $\Gr$ is a sheaf.
\end{thm}

\begin{proof}
  It is clear that $\Gr$ is small and that $a: \Gr\to\PrSh(\elGr)$ is fully
  faithful and dense.  The nerve theorem will be an instance of the general
  nerve theorem of Weber~\cite{Weber:TAC18} (Theorem~4.10), if just we can
  establish that $a:\Gr\to \PrSh(\elGr)$ provides arities for the free-properad
  monad, which temporarily we denote by $T$.  By
  Berger-Melli\`es-Weber~\cite{Berger-Mellies-Weber:1101.3064} (Propositions
  2.12--2.14), to say that $a$ provides arities for $T$ is equivalent to saying
  that the natural functor $a \comma Ta \comma T \to a \comma T$ given by
  composition has connected fibres.  The objects in $a \comma T$ are maps
  $R\to \ov Y$, where $R$ is a graph and $Y$ is an arbitrary presheaf.
  The fibre over a map $R \to \ov Y$ is
  the category of factorisations
    $$\xymatrix{
     R \ar[rr]\ar[rd] && \ov Y \\
     & \ov Q \ar[ru] & 
  }$$
  such that the middle object is a graph, and the second map is free.
  But having (weak) generic factorisation say precisely that this factorisation
  category has a weakly initial object, and in particular is connected.
\end{proof}

\begin{blanko}{Remarks on the proof.}
  Weber established the general nerve theorem (\cite{Weber:TAC18}, Theorem~4.10)
  in the situation where $T$ is a monad on a category $\CC$ and $a:\Theta_0 \to
  \CC$ provides arities for $T$.  (To provide arities means that a certain left
  Kan extension is preserved by the monad.)  He showed furthermore 
  (\cite{Weber:TAC18}, Proposition~4.22)
  that if $\CC$ is a
  presheaf category and $T$ admits strict generic factorisations, then there is
  a canonical choice of $\Theta_0$, namely the full subcategory spanned by the
  objects that appear as middle objects of generic/free factorisations of maps from
  a representable to the terminal presheaf.  It
  was observed in \cite{Kock:0807.2874} (Remark~2.2.11) that the arguments in Weber's
  proof in
  fact yield the more general criterion: $a$ provides arities for $T$ if the
  natural functor $a \comma Ta \comma T \to a \comma T$ has connected fibres and
  admits a section.  Weber (personal communication) pointed out that in fact the
  section is not necessary (although of course in practice the section is often
  provided by generic factorisations).  Finally,
  Berger-Melli\`es-Weber~\cite{Berger-Mellies-Weber:1101.3064} (Propositions 
  2.12--2.14) turned the
  criterion into an if-and-only-if statement, and gave a more conceptual formulation
  and a more elegant proof, as part of a more streamlined overall treatment.
\end{blanko}

\begin{blanko}{Remarks on weak versus strict generic factorisations.}  
  Weber's original notion of generic morphism was the weak notion
  \cite{Weber:TAC13}, which is the one relevant in the present work.  Subsequent
  work \cite{Weber:TAC18}, \cite{Kock:0807.2874},
  \cite{Berger-Mellies-Weber:1101.3064} focused on the strict notion, which is
  intimately related to the notion of local right adjoint.
  The weak/strict distinction is closely related with
  the distinction between analytic and 
  polynomial functors, which in fact was Weber's motivation for introducing
  the notions of generic map in the first place~\cite{Weber:TAC13}.
  
  Although
  it is not a precise result at the time of this writing, it seems that in
  practice the weak situation (related to weakly cartesian monads) always arises
  from truncation of a strict situation in a homotopical setting, a monad which
  is cartesian in the homotopical sense.  
  This principle transpires from joint work with David Gepner~\cite{Gepner-Kock}
  developing the theory of polynomial functors and
  generic factorisations in $\infty$-categories, and observing in particular
  that in the $\infty$-world, the difference between analytic and polynomial 
  evaporates.  (This and some related results are previewed in
  \cite{Kock:MFPS28}.) 
  
  The present case seems to corrobate this principle.  The free-properad monad
  is only weakly cartesian, due to the presence of the $\pi_0$ in the formula
  for it.  In Section~\ref{sec:prpd-grpd} below, a groupoid-valued version of
  the monad is described which avoids this truncation.  I claim that the
  groupoid version of the free-properad monad is cartesian and is a local right
  adjoint, and that it therefore has strict generic factorisations (all in the
  homotopy sense of \cite{Gepner-Kock}).  
\end{blanko}

\subsection{Working in the category $\wtil\Gr$.}

\label{sec:wtilGr}

The category $\wtil\Gr$ is meant to contain all the combinatorics of graphs
relevant to properad theory.  The subcategory $\Gr$ already has the `geometric
part': open inclusions, etale maps, symmetries, colimits.  (In the following
when we talk about colimits they are understood to be in $\Gr$.)  The new maps
introduced, the refinements, represent the algebraic structure, embodying the
substitution aspects.  It is an important feature of the present approach that
this category in which the two aspects interact is generated by general
machinery (such as presheaves and monads).  While the abstract description as a
restricted Kleisli category was enough to establish the generic factorisations
and the nerve theorem, it is worthwhile, as we do in this subsection, to extract
more explicit descriptions of the refinement maps, and how they interact with
the etale maps.  

\begin{lemma}
  Any map $R\to Y$ in $\wtil\Gr$ sends edges to edges.
\end{lemma}
\begin{proof}
  Indeed, the map is given by a functor $\el(R) \to \Gr\comma Y$, assumed to be
  a residue-compatible graph of graphs in the technical sense of \ref{rcgg}, and in
  particular it sends $A$-objects of $\el(R)$ to unit $Y$-graphs, which is
  the same as saying that it sends edges to edges.
\end{proof}

\begin{blanko}{Hom sets of refinements.}
  Let $R$ and $Y$ be graphs.  The set of refinement maps from $R$ to $Y$ is the
  set of isoclasses of functors $\el(R) \to \Gr\comma Y$, that are
  residue-compatible graphs of $Y$-graphs, and with the property that the colimit is
  terminal.  Since all those graphs $G_x$ map into $Y$ as open inclusions,
  instead of
  calculating the colimit in $\Gr\comma Y$, we can calculate it in the
  poset $\Sub(Y)$.  Here there are no isomorphisms, so we can say that
  the set of refinements $R \to Y$ is the set of residue-compatible 
  graphs of subgraphs-of-$Y$
  $$
  \big\{ \gamma:\el(R) \to \Sub(Y) \mid \gamma \text{ rcgg, } \colim(\gamma) = Y \big\} .
  $$

  In the special case where $R = C^m_n$, the unique node must be sent to the
  subgraph $Y$ itself, so the only choice involved is where to send the edges,
  which amounts to specifying an isomorphism $C^m_n \simeq \res(Y)$.
  So when this is possible at all ($Y$ has the correct residue), there are
  $m!n!$ elements in the hom set.
  
  On the other hand, for $Y$ fixed, we can describe the set of all possible
  refinements $R \to Y$, with variable $R$.  They are given precisely by the
  bijective-on-nodes open covers of $Y$, i.e.~collections of open subgraphs of
  $Y$ such that each node is in precisely one subgraph.  Such an open cover is a
  gluing datum, and $R$ is the indexing graph of it.  (This cover interpretation
  of generic maps was used by Berger~\cite{Berger:Adv} in a more general context
  (see also \cite{Kock:0807.2874}), and can be seen as a historical precursor to
  the notion of generic map.)
\end{blanko}

\begin{lemma}
  A refinement is completely determined by its values on edges.
\end{lemma}

\begin{proof}
  Given a refinement $R \to Y$, for each node $x$ in $R$ we have the canonical
  neighbourhood which is a corolla in $R$.  We can restrict the refinement to
  each of these corollas (see \ref{refactor} below for details), and in each
  case, by the previous paragraph, the refinement is determined by its values on
  edges.  Hence  also the whole map is determined by
  its values on edges.
\end{proof}

\begin{blanko}{Remark.}
  We also observed (\ref{GGAA}) that an etale map is determined by its values on
  edges, except when the domain has no edges.  Even with this exception, this does not
  imply that a general map in $\wtil\Gr$ is determined by its value on edges,
  because the edge map does not determine the factorisation.  (For examples, see
  \cite{Hackney-Robertson-Yau}.)
\end{blanko}

\begin{blanko}{Composition in $\wtil\Gr$.}
  Given maps in $\wtil\Gr$
  $$
  R \to Q \to Q'
  $$
  where $R\to Q$ is given by a diagram $\gamma: \el(R) \to \Gr\comma Q$
  sending a node $x$ in $R$ to some etale map $G_x \to Q$, with a chosen 
  isomorphism $\res(x) \simeq \res(G_x)$,
  and where $Q\to Q'$ is given by a diagram $\delta: \el(Q) \to \Gr\comma Q'$.
  Then
  the composite map $R\to Q'$ is described as follows.
  Put
  $$
  G'_x := \colim\big( \el(G_x) \to \el(Q) \stackrel\delta\to \Gr\comma Q' \big).
  $$
  These graphs are the
  ingredients of the new diagram $\gamma': R \to \Gr\comma Q'$,
  which defines the composite map.
  (Essentially we are just saying that a colimit indexed by a colimit can be
  expressed as a single colimit, and basically we are just repeating the
  associativity argument.)
\end{blanko}

\begin{blanko}{Refactoring etale/refinement as refinement/etale.}\label{refactor}
  As a special case, given an etale map $R'\to R$ and a refinement map $R \to Q$ defined by 
  $\gamma:\el(R) \to \Gr$ (with colimit $Q$), then in the diagram
  $$\xymatrix{
     R' \ar[r]^{\text{etale}}\ar@{..>}[d]_{\text{refine}} & R \ar[d]^{\text{refine}} \\
     Q' \ar@{..>}[r]_{\text{etale}} & Q
  }$$
  put $Q' := \colim\big( \el(R') \!\to\! \el(R) \!\stackrel\gamma\to \!\Gr\big)$.
  By construction this defines a refinement $R' \to Q'$, and an etale map $Q'\to 
  Q$ is induced from the description of $Q'$ as a 
  colimit.
\end{blanko}

\begin{prop}\label{po}
  Given an open inclusion $H \to G$ and a refinement $H \to Q$,
the pushout
$$\xymatrix @!@+4pt {
   H \ar[r]^{\text{open}}\ar[d]_{\text{refine}} & G \ar@{..>}[d]^{\text{refine}} \\
   Q \ar@{..>}[r]_{\text{open}} & P\ulpullback
}$$
exists in $\wtil\Gr$, and the dotted maps are again an open inclusion and a 
refinement as indicated.  The pushout is calculated by {\em identity extension} 
(see proof),
and in particular there is a natural isomorphism
$P\shortsetminus Q \simeq G\shortsetminus H$.
\end{prop}

\begin{proof}
  If $H \to Q$ is given by $\el(H) \to \Gr$ with colimit $Q$,
  we need to extend to a functor $\gamma : \el(G) \to \Gr$ and define $P$ to be 
  its colimit.  Simply assign to each node $x$ in the complement of $H$
  the graph given by the canonical neighbourhood of $x$ in $G$.
  The colimit description provides the etale map $Q \to P$, and from 
  \ref{refactor} it is clear that the resulting square commutes.
  Alternatively, just as $G$ is obtained by gluing $H$ to its complement 
  $\compl H$ along $S:= H \cap \compl H$, the new graph $P$ is obtained by gluing 
  $Q$ to $\compl H$ along $S$.  This makes sense canonically since $S$ is a 
  subset of the set of ports of $H$, and since $Q$ 
  and $H$ have the same ports.
  From this description, it is clear that $Q\to P$
  is an open inclusion again.  It remains to check
  that the square is a pushout, but again this follows from the 
  construction of $P$: given another commutative square with the same solid part
  and with a different $P'$ instead of $P$, we need to establish that there is 
  a unique map $P \to P'$ making every everything commute.  To give this map
  is to give $\el(P) \to \Gr\comma P'$, and it is readily seen that there is a 
  unique such functor, since the nodes in $P$ are identified with the nodes in 
  $Q$ plus the nodes in $G\shortsetminus H$.
\end{proof}
(Note that it is not true in general that etale maps allow pushouts along 
refinements.)

\begin{blanko}{Graph substitution (cf.~\cite{Yau-Johnson},~\cite{Hackney-Robertson-Yau}.})
  In the situation of Proposition~\ref{po}, if $H$ consists of a single node 
  $x$,
  then $Q$ is a graph with $\res(Q) = \res (x)$, and $P$ is the result of
  {\em substituting $Q$ into the node $x$ of $G$}.
\end{blanko}

\begin{lemma}\
  In the situation of Proposition~\ref{po},
  if $H \to G$ is convex, then $Q \to P$ is convex.
\end{lemma}
\begin{proof}
  This follows immediately from $P\shortsetminus Q \simeq G\shortsetminus H$,
  together with the complement characterisation of convexity (\ref{convcomp}).
\end{proof}

\begin{cor}\label{refinenode}
  If an open subgraph arises from refinement of a single node, then it is convex.
\end{cor}
Conversely:
\begin{lemma}\label{conv=ref}
  If $Q \subset P$ is a convex open subgraph (of an acyclic graph $P$), then there exists an (acyclic) 
  graph $G$
  with a node $x$ and a refinement $x\to Q$ yielding $Q \subset P$ by pushout.
\end{lemma}
\begin{proof}
  If $G$ and $x$ exist, we must have $G\shortsetminus x = P \shortsetminus Q$.
  Put $S := Q \cap \compl Q$, then $P$
  is the gluing 
  $$
  S \rightrightarrows Q + \compl Q \to P.
  $$
  Since $S \subset \operatorname{ports}(Q) = \operatorname{ports}(\res Q)$,
  we can glue in $\res (Q)$ instead of $Q$, obtaining $G$ in this way:
  $$
  S \rightrightarrows \res (Q) + \compl Q \to G.
  $$
  It remains to see that $G$ is acyclic --- this is where convexity of $Q$ 
  comes in: a wheel in $G$ through $x$ would induce a path in $G\shortsetminus 
  x = P \shortsetminus Q$
  from an edge in $S$ to another edge in $S$.  But this is impossible since $Q$ 
  is convex (\ref{convcomp}).  (And of course $G$ cannot contain a wheel not 
  through $x$, since 
  they would also be a wheel in $G\shortsetminus x = P \shortsetminus Q \subset 
  P$.)
\end{proof}
(Note that $G$ might be an inner edge in $Y$; then the complement is
not a subgraph.)

\medskip

Corollary~\ref{refinenode} and Lemma~\ref{conv=ref} together are also
established in \cite{Hackney-Robertson-Yau}, Theorem 5.38, modulo
set-up and terminology.

\begin{lemma}\label{thruconvex}
  In the situation of \ref{refinenode} and \ref{conv=ref},
  $$\xymatrix @!@+4pt {
   C \ar[r]^{\text{open}}\ar[d]_{\text{refine}} & G \ar[d]^{\text{refine}} \\
   Q \ar[r]_{\text{open}} & P\ulpullback
}$$
 where $C$ is a corolla, suppose $x$ and $y$ are edges in $P\shortsetminus Q =
 G\shortsetminus C$.  If there is a path in $P$ from $x$ to $y$, then there
 is also a path in $G$ from $x$ to $y$.
\end{lemma}
\begin{proof}
  If the path is disjoint from $Q$ it is also a path is $G$.  Otherwise, 
  since $Q$ is convex, the path cannot enter and leave $Q$ twice.  So it goes in
  three steps: first from $x$ to $x'\in \operatorname{im}(Q) = \operatorname{im}(C)$, second from
  $x'$ to $y'$ inside $Q$, and third from $y' \in 
  \operatorname{ex}(Q)=\operatorname{ex}(C)$ to $y$.
  Now there is clearly also a path in $C$ from $x'$ to $y'$, so by concatenation of
  paths there is a path in $G$ from $x$ to $x'$ to $y'$ to $y$.
\end{proof}

\begin{prop}\label{convref-refconv}
  In any commutative square in $\wtil\Gr$
  $$\xymatrix @!@+4pt {
   H \ar[r]^{\text{etale}}\ar[d]_{\text{refine}} & G \ar[d]^{\text{refine}} \\
   Q \ar[r]_{\text{etale}} & P  ,
  }$$
  if $H \to G$ is a convex inclusion then $Q \to P$ is a convex inclusion.
\end{prop}
In other words, refactoring etale/refinement to refinement/etale as in
\ref{refactor}, takes convex/refinement to refinement/convex.
\begin{proof}
  The square factors vertically as
  $$\xymatrix @!@+6pt {
   H \ar[r]^{\text{convex}}\ar[d]_{\text{refine}} & G \ar[d]^{\text{refine}} \\
   Q \ar[r]_{\text{convex}}\ar[d]_{\text{id}} & P' \ulpullback 
   \ar[d]^{\text{refine}} \\
   Q \ar[r]_{\text{etale}} & P .
  }$$
  Here the vertical maps in the top square refine nodes in $H$, and the middle
  map is convex by Proposition~\ref{po}.  The vertical maps in the bottom square
  refine nodes outside $Q$.  Therefore $Q \to P$ is an open inclusion since $Q 
  \to P'$ is.  Suppose there were a path in $P$ violating convexity
  of $Q \to P$.  Then by iterated use of Lemma~\ref{thruconvex} there would also
  be a path in $P'$ violating the convexity of $Q \to P'$.
\end{proof}

\begin{blanko}{Hackney-Robertson-Yau category.}\label{hry-cat}
  Proposition~\ref{convref-refconv} is essentially equivalent to Lemma~5.50 of
  \cite{Hackney-Robertson-Yau}, modulo set-up and terminology.  It
  follows from the proposition (together with Lemma~\ref{convcomp}) that we can
  obtain a subcategory of $\widetilde \Gr$ by making the following restriction
  on the maps: allow only maps whose free part is a convex open inclusion.  This
  is the Hackney-Robertson-Yau category $\Gamma$ of connected acyclic graphs
  \cite{Hackney-Robertson-Yau}.
  
  Note that the refinement/convex factorisations that exist in $\Gamma$ by
  construction are unique up to {\em unique} isomorphism, simply because convex
  open inclusions are mono\-morphisms.  Hence the class of refinements and the
  class of convex open inclusions form an orthogonal factorisation system in
  $\Gamma$.
\end{blanko}

\section{Hypergraphs}

In a nutshell, the idea is this: the free-properad monad applied to a
(bi-collection represented by a) graph $X$ is a bi-collection which
is not again a graph.  Nevertheless, intuitively it should be represented by a
diagram
\begin{equation}\label{eq:monadformula}
\xymatrix @! {
A & \ar[l] \et^1 (X) \ar[r] & \et (X) & \ar[l] \et_1 (X) \ar[r] & A ,
}
\end{equation}
where $\et(X)$ consists of etale maps from graphs to $X$, and $\et^1 (X)$
(resp.~$\et_1 (X)$) consists of etale maps to $X$ with a marked import
(resp.~export).  In other words, the monad promotes all `subgraphs' to being
nodes in their own right.  With this proliferation of nodes, it is no longer
true that an edge is incoming (or outgoing) of at most one node; in other
words, the injectivity axiom is violated and the new structure is no longer a
graph.  The intuition is that it is instead a {\em directed hypergraph}.  To
formalise these ideas, one further ingredient is needed, namely to use groupoids
to correctly deal with automorphisms of etale coverings (deck transformations): for
the statement to be correct we must use {\em groupoid-enriched} hypergraphs.
Specifically, we need $\et(X)$ to be the {\em groupoid} of all etale maps to
$X$, not just the set of iso-classes of such.

The main result of this section, Theorem~\ref{thm:HGr}, states that
{\em
  the free properad on a hypergraph is again a hypergraph (given by 
  \eqref{eq:monadformula})}.

\subsection{Discrete hypergraphs}

The theory of hypergraphs is a extensive research topic, with a variety of
different
applications in computer science.  A standard text book
on hypergraphs is Berge~\cite{Berge:1989}.  For the notion of directed
hypergraph, a classical reference is \cite{Gallo-Longo-Pallottino-Nguyen}.
Here we take a novel approach to
directed hypergraphs, englobing naturally the theory of directed graphs above.

\begin{blanko}{Directed hypergraphs.}\label{hyp}
  A {\em directed hypergraph} is a diagram of sets
  $$
  \xymatrix{
  A & \ar[l]_s I \ar[r]^p & N & \ar[l]_q O \ar[r]^t & A
  }
  $$
  for which both $I \to A \times N$ and $O \to N \times A$ are relations
  (i.e.~are injective maps). 
  The elements in $N$ are called {\em nodes}; 
  the elements in $A$ are called {\em hyperedges}.  A hyperedge connects 
  one set of nodes to another set of nodes.
  A directed hypergraph can be represented by two incidence matrices between
  nodes and hyperedges. 
  A directed hypergraph is called {\em loopfree} if
  these two relations are disjoint, i.e.~if also $I+O \to A \times N$ is a 
  relation.  (This simply means that a hyperedge cannot contain
  the same node in its domain and in its codomain.)  Loopfree hypergraphs 
  are what are called hypergraphs in \cite{Gallo-Longo-Pallottino-Nguyen},
  where they are encoded as a single signed incidence matrix.
  For the present purposes it is essential to allow loops.
  
  From now on we simply say {\em hypergraph} for directed hypergraph.

  A graph is a hypergraph, since a diagram~\eqref{eq:graph} clearly
  satisfies the conditions: if $s$ and $t$ are themselves injective, clearly
  the two spans are relations. 
  For the present purposes, a fruitful interpretation of the hypergraph 
  axiom, is that a hypergraph is {\em locally} a graph, in the sense that
  for each node $x$, the
  maps $I_x \to A$ and $O_x\to A$ are injective.  In fact conversely,
  if for every node $x$ the maps $I_x \to A$ and $O_x\to A$ are injective
  then $X$ is a hypergraph.  Indeed, $I = \sum_{x\in N} I_x$ and $A\times N
  = \sum_{x\in N} A$, and the map $I \to A\times N$ is just the sum of all
  the maps $I_x \to A$.  Similarly for $O$.

  The notion of hypergraph has a self-duality, in the sense that interchanging
  the role of nodes and hyperedges yields again a hypergraph.  However, the study
  of hypergraphs is biased, so that all notions are geared towards the embedding
  of graphs inherent in the choice of symbols.
  With the asymmetry in mind, we define classes of morphisms as follows.  A
  morphism is a diagram like \eqref{eq:map}; it is {\em etale} if the middle
  squares are pullbacks, and an {\em open inclusion} if it is level-wise
  injective and etale.  (Note that according to this definition, etale maps are
  arity-preserving for nodes, but not necessarily on hyperedges.)

  Observe that we have not required the sets to be finite.  This is 
  because the free \mbox{properad} on a graph may be an infinite hypergraph
  (just as the free category on a (closed) directed graph may be an infinite category).
  The relevant finiteness condition is just that $p$ and $q$ be finite maps.
  These are called hypergraphs of {\em finite type}.  Henceforth we only consider
  hypergraphs of finite type.
\end{blanko}

\begin{blanko}{The core of a hypergraph.}
  The {\em core} of a hypergraph $X$, denoted by $X^\bullet$, is the (possibly
  infinite) closed graph $O\times_A I \rightrightarrows N$ given by the diagram
  $$
\xymatrix@!@-30pt{&& \ar[ld] O \times_A I \ar[rd] && \\
& \ar[ld]_q O \ar[rd]^t && \ar[ld]_s I \ar[rd]^p & \\
N & & A & & N   .
}
$$
Just as for graphs, we have
\begin{prop}
  Taking core is right adjoint to the inclusion of (possibly infinite) closed graphs into 
  hypergraphs.
\end{prop}
While for $X$ a graph, the core amounts to deleting all ports, for $X$ a general
hypergraph, taking core involves furthermore replacing every hyperedge with a 
number of edges, one for each connection it realises.  The counit is not in
general injective.
\end{blanko}
\begin{blanko}{Inner edges of a hypergraph.}
  The set of {\em inner edges} of a hypergraph is by definition the set of edges of its core.
  Hence the set of inner edges is $O\times_A I$.  Note that the canonical map
  $O\times_A I \to A$ from inner edges to hyperedges is not in general 
  injective.
\end{blanko}

\begin{blanko}{Dual embedding of digraphs.}\label{dual}
  When directed graphs in the classical sense (pre\-sheaves on
  $\bullet\!\leftleftarrows\!\bullet$) are used as the structures that
  underlie or generate categories, the nodes play the role of objects (and are
  not modified by the free-category monad) and the edges generate the arrows.
  This is in contrast with the free-operad and free-properad monads, where the
  edges are left unmodified, and the nodes generate the operations.  The
  contrast is accounted for elegantly  by hypergraphs: there is a dual embedding of
  (possibly infinite) 
  classical directed graphs into hypergraphs, sending $E \rightrightarrows V$ to
  the hypergraph (of finite type)
  $$
  \xymatrix{
  V & \ar[l]_s E \ar[r]^= & E &\ar[l]_= E \ar[r]^t & V .
  }$$
  In other words, it interprets edges (arrows) as nodes (operations), and interprets
  vertices as hyperedges.
\end{blanko}

\begin{prop}
  This dual embedding has a right adjoint, sending a hypergraph $AINOA$ to
  $$
  I\times_N O \rightrightarrows A .
  $$
\end{prop}

\begin{blanko}{Sums, connectedness, acyclicity, loops.}
  The notions of connectedness and acyclicity are defined in the same way for
  hypergraphs as for graphs, but do not play an important role for the present
  purposes, as the free properad on a connected acyclic graph is a hypergraph 
  which may be neither connected nor acyclic.  (Specifically, if a graph $X$
  has no ports, then $\ov X$ will contain a corresponding isolated node, while
  each edge in a graph $X$ will become a node in $\ov X$ with that same
  edge as a loop.)
\end{blanko}

\begin{blanko}{Canonical neighbourhood and canonical cover.}
  The notions of canonical neighbourhood of a node and canonical etale cover
  are the same for hypergraphs as for graphs:
  given a subset of nodes $N'\subset N$, we construct
$$
\xymatrix{
I'+O' \ar[d]_\alpha& \ar[l] I'\ar[d] \drpullback\ar[r] & N'\ar[d] & 
\ar[l]\dlpullback O'\ar[d] 
\ar[r] & I'+O'\ar[d]^\alpha \\
A & \ar[l] I \ar[r] & N & \ar[l] O \ar[r] & A
}
$$
This is clearly a disjoint union of corollas, each of which is the canonical 
etale neighbourhood of a node in $X$.
For $N'=N$, these jointly cover $X$.  It is the {\em canonical etale cover}.
\end{blanko}

\begin{prop}
    The category of hypergraphs and etale maps
    admits pushouts and coequalisers of shrub injections.
\end{prop}
Note that contrary to the case of graphs, there are no import-export conditions:
one can glue any hyperedge to any other hyperedge.

\begin{proof}
    In the presheaf category of diagrams of shape $AINOA$,
    pushouts and coequalisers
    are computed level-wise.  It is enough to prove that the results
    are hypergraphs again.  We do pushouts, the case of coequalisers
    being analogous.  Given hypergraphs $AINOA$
    and $A'I'N'O'A'$, and maps from the shrub $S000S$,
    the hyperedge set of the pushout is the amalgamated sum 
    $A +_S A'$. We need to show that $I+I' \to 
    (A+_S A' )\times (N+N')$ is injective (and similarly with $O+O'$).
    Since the originals are hypergraphs, we have $I+I' \into (A\times 
    N) + (A'\times N')$.  
    But we also have $(A\times N) + (A'\times N') \into (A+_SA') \times (N+N')$ 
    by the distributive law (since $N$ and $N'$ are disjoint in $N+N'$),
    as seen in this figure:
    \begin{center}
	\begin{texdraw}\setunitscale 0.9
    \move (0 0) 
    \bsegment
    \bsegment 
      \move (0 0) \lvec (30 0) \lvec (30 35) \lvec (0 35) \lvec (0 0) 
      \htext (15 16.5){\footnotesize $A\!\times\!N$}
    \esegment
    \move (20 35)
    \bsegment 
      \move (0 0) \lvec (40 0) \lvec (40 20) \lvec (0 20) \lvec (0 0)
      \htext (20 10){\footnotesize $A'\!\times\!N'$}  
    \esegment
          \esegment
	  \move (0 0) \lvec (60 0) \lvec (60 55) \lvec (0 55) \lvec (0 0)
      \htext (30 -7){\footnotesize $A+_S A'$}
      \htext (-20 27.5){\footnotesize $N+N'$}
\end{texdraw}
\end{center}
\end{proof}

\begin{blanko}{Elements.}
  The {\em category of elements} of a hypergraph $X$ is defined exactly as for 
  graphs:
  $$
  \el(X) := \elGr \comma X ,
  $$
  the category of elements of the presheaf
  \begin{eqnarray*}
    \elGr\op & \longrightarrow & \Set  \\
    E & \longmapsto & \Hom(E,X) .
  \end{eqnarray*}
  Just as in the case of graphs, $\el(X)$ is naturally equivalent to the 
  category whose object set is $A+N$ and whose set of non-idenity arrows is
  $I+O$.  Again there is a canonical functor $\el(X) \to \Gr \subset \HGr$
  given by sending each $A$-object to the unit graph $U$, and sending each object
  $x\in N$ to the canonical neighbourhood of $x$.
\end{blanko}

\begin{lemma}\label{hyper-density}
  Every hypergraph $X$ is a colimit of its elements (which are elementary
  graphs).  Precisely,
  $$
  X = \colim \big( \el(X) \to \Gr \subset \HGr \big).
  $$
\end{lemma}

\begin{proof}
  The colimit can be computed as an iterated coequaliser over edges.
\end{proof}

\subsection{Groupoid-enriched hypergraphs}

\begin{blanko}{Groupoids.}
  We shall freely use basic facts about groupoids, and in particular the
  consistent homotopy approach.  See
  G\'alvez-Kock-Tonks~\cite{GalvezCarrillo-Kock-Tonks:1207.6404}, Section 3,
  where there is some introduction.  The important feature is that all notions
  are up-to-homotopy: in particular, by commutative square is meant a square
  with a specified $2$-cell (a homotopy), and pullback means homotopy pullback.
  If the bottom corner in a commutative square is just a set,
  then `commutative' has its usual meaning and homotopy pullback is the same
  thing as ordinary pullback.  A special case of homotopy pullback is homotopy
  fibre, which is (homotopy) pullback to a point.  We shall also need homotopy quotients
  (also called action groupoid or semi-direct product), in the situation where a
  group acts on a set or on a groupoid: where the naive quotient identifies $x$
  with $x.g$, the homotopy quotient rather sews in a path between $x$ and $x.g$.
  The naive quotient is obtained by taking $\pi_0$ of the homotopy quotient.
  
  From now on, all pullbacks, fibres, quotients, (and more generally) limits 
  and colimits refer to the homotopy notions.
\end{blanko}

\begin{blanko}{Groupoid-enriched hypergraphs.}
  A {\em groupoid-enriched hypergraph} is a diagram of groupoids
  $$
  \xymatrix{
  A & \ar[l]_s I \ar[r]^p & N & \ar[l]_q O \ar[r]^t & A
  }
  $$
  satisfying the following three conditions.
  \begin{enumerate}
    \item $A$, $I$ and $O$ are discrete (i.e.~equivalent to sets), and $N$ is 
    locally finite (i.e.~has finite vertex groups).
  
    \item $I \to N \leftarrow O$ are discrete fibrations.  (The hypergraph is 
    called of finite type if these fibres are finite.  This will always be 
    assumed below.)
  
    \item Both $I \to A \times N$ and $O \to N \times A$ are monomorphisms.
  \end{enumerate}

  Observe that from condition 1 and 2 we have the diagram
  $$\xymatrix{
     I \ar[d]\ar[rd]^{\text{discr.fib}} &  \\
     A\times N \ar[r]_{\text{discr.fib}} & N
  }$$
  and therefore also $I \to A\times N$ is automatically a discrete fibration,
  i.e.~has discrete (homotopy) fibres.  Condition 3 says that these discrete fibres are
  either singleton or empty.
  
  Observe that since $I$ (resp.~$O$) is discrete, and also the fibres $I_x$ 
  (resp.~$O_x$) are discrete, the map $I\to N$ (resp.~$O\to N$) is in fact
  a disjoint union of torsors.  More precisely, for each $x\in N$, 
  the vertex group $\Aut_N(x)$ acts freely on $I_x$ (resp.~on $O_x$). This means
  that locally at each node $x$, the hypergraph is a `stacky corolla', as 
  detailed below.
\end{blanko}

\begin{blanko}{Etale maps.}
  An {\em etale map} of groupoid-enriched hypergraphs is a diagram
  $$
  \xymatrix{
  A' \ar[d]_\alpha& \ar[l] I'\drpullback \ar[d] \ar[r] & N'\ar[d] & \ar[l] 
  \dlpullback O'\ar[d] 
  \ar[r] & A'\ar[d]^\alpha \\
  A & \ar[l] I \ar[r] & N & \ar[l] O \ar[r] & A
  }   
  $$
  in which the middle squares are (homotopy) pullbacks.  We denote by $\HGr$ the
  category of (groupoid-enriched, finite-type) hypergraphs and etale maps.
\end{blanko}

\begin{blanko}{Stacky corollas.}
  Consider a corolla $C^m_n$:
  $$
  m+n \leftarrow m \to 1 \leftarrow n\to m+n
  $$
  and suppose that a finite group $G$ acts freely on $m$ and freely on $n$.
  Then we can form the levelwise homotopy quotient, which receives an etale map from 
  $C^m_n$:
  $$\xymatrix{
     m+n \ar[d] & \ar[l] m \drpullback \ar[r]\ar[d] & 1 \ar[d] & \ar[l] n 
     \dlpullback \ar[d]\ar[r] & 
     m+n \ar[d] \\
     \frac{m+n}{G}  & \ar[l] \frac{m}{G} \ar[r] & \frac{1}{G} & \ar[l] 
     \frac{n}{G} \ar[r] & \frac{m+n}{G}
  }$$
  The result is a hypergraph $C^m_n/G$ called a {\em stacky corolla}.
  It arises as the homotopy coequaliser 
  $$
  C^m_n \times \un G \rightrightarrows C^m_n \to C^m_n/G,
  $$
  where $\un G$ denotes the discrete set of elements in the group $G$,
  and $C^m_n \times \un G$ denotes the disjoint union of that many
  copies of $C^m_n$.
\end{blanko}

\begin{blanko}{Hypergraphs as colimits of elementary graphs.}\label{denshyp}
  We need the groupoid version of the result that {\em every hypergraph $X$ is
  the colimit of its elements}.  This is true again: the category of elements
  $\elGr\comma X$ can be described explicitly as having object set $A +
  \operatorname{obj}(N)$ and arrow set $I+O+\operatorname{arr}(N)$: an arrow
  $f\in I$ has domain $s(f)$ and codomain $p(f) \in \operatorname{obj}(N)$,
  while an arrow $g\in O$ has domain $t(g)$ and codomain $q(g) \in
  \operatorname{obj}(N)$.  The category of elements is the domain of a canonical
  diagram $\el(X)\to\HGr$, whose colimit is $X$.  This colimit can be computed
  as an iterated {\em strict} coequaliser of stacky corollas over shrubs, and
  then the stacky corollas in turn are {\em homotopy} quotients as above.
\end{blanko}

Let $\operatorname{cor}(X) := \Cor\comma X$ denote the groupoid of etale maps
from corollas into $X$.  Let $\operatorname{cor}^1(X)$ denote the groupoid of
such maps but with a marked import, and let $\operatorname{cor}_1(X)$ denote the
groupoid of such maps but with a marked export.  (Compare \ref{barhyp}.)

\begin{lemma}
  For $X$ a hypergraph $AINOA$, there is a natural equivalence of groupoids 
  $$
  \operatorname{cor}(X) \simeq N
  $$
  Similarly, there are natural equivalences $\operatorname{cor}^1(X) 
  \simeq I$
  and $\operatorname{cor}_1(X) \simeq O$, as well
  as a bijection $\operatorname{hedge}(X) := \Map(U,X) \simeq A$.
  Altogether, $X$ is equivalent to
  $$
  \xymatrix @! {
   \operatorname{hedge}(X)& \ar[l] \operatorname{cor}^1 (X) \ar[r] & \operatorname{cor} (X) &
  \ar[l] \operatorname{cor}_1 (X) \ar[r] & \operatorname{hedge}(X) .
  }
  $$
\end{lemma}
\begin{proof}
  Given $x\in N$ we pull back (and use sum injections) to get an etale map
    $$\xymatrix{
     m+n \ar[d] & \ar[l] m \drpullback \ar[r]\ar[d] & 1 \ar[d] & \ar[l] 
     n
     \dlpullback \ar[d]\ar[r] & 
     m+n \ar[d] \\
     A  & \ar[l] I \ar[r] & N & \ar[l] 
     O\ar[r] & A
  }$$
  where $m$ is the cardinality of $I_x$ and $n$ is the cardinality of $O_x$.
  There are $m!$ possible bijections $m\isopil I_x$ and $n!$ possible bijections
  $n \isopil O_x$, but they all yield isomorphic etale maps.  The automorphism
  group of a fixed such map gets no contribution from the $I$ and $O$ level,
  since these are discrete.  The only contribution to automorphisms comes from the
  automorphisms of $x:1 \to N$, and these form precisely the vertex group of
  $x$.  The statements for the remaining sets are straightforward.  Note that
  the first statement is precisely a consequence of the hypergraph axioms.  (In
  fact the condition is {\em equivalent} to the hypergraph axioms!)
\end{proof}

\subsection{The free properad on a groupoid-valued bi-collection}

\label{sec:prpd-grpd}
The constructions in \ref{sec:freeproperad} can be carried out with 
coefficients in groupoids instead of coefficients in sets (and in fact
this is in a sense more natural, as we already used groupoids in the 
constructions, and now avoid taking $\pi_0$ in the end).
This means that we keep the $1$-categories $\elGr$ and $\Gr$, but consider
presheaves with values in $\Grpd$:
$$
\PrSh(\elGr) := \Fun(\elGr\op,\Grpd) .
$$
These could be called prestacks instead of presheaves, as in fact we allow
pseudofunctors, but in keeping with the philosophy that the real thing is
$\infty$-groupoids, where the appropriate weakenings are taken care
of automatically by the formalism, and that $\infty$-groupoids are regarded as
a fancy version of sets, we stick to the presheaf terminology, and refrain
also from going into subtle distinctions between functors and pseudofunctors.
As in the set case we have that presheaves (prestacks) on $\elGr$
are naturally equivalent to sheaves (stacks) on $\Gr$, so that a presheaf
on elementary graphs can be evaluated also on general connected graphs 
by the homotopy limit
formula
$$
F[G] = \lim_{E\in \el(G)} F[E].
$$

Similarly, the definition of the free properad 
\begin{align*}
   \PrSh(\elGr) \ \longrightarrow & \ \PrSh(\elGr) \\  
   F \ \longmapsto & \ \ov F ,
\end{align*}
now uses a homotopy colimit:
\begin{align*}
\ov F[m,n] \ := & \ \colim_{G \in (m,n)\kat{-Gr}_{\iso}} F[G] \\[6pt]
=& \ \sum_{G\in \pi_0((m,n)\kat{-Gr}_{\iso})}  \wfrac{F[G]}{\Aut_{(m,n)}(G)} \\[6pt]
=& \  (m,n)\kat{-Gr}_{\iso} \comma F  .
\end{align*}
Here the double-line denotes the homotopy quotient.

(The resulting functor is a pseudo-monad rather than a strict monad, and the
notion of groupoid-enriched properad should be that of pseudo-algebra for this
pseudo-monad.)

The relationship with the properad monad given in \ref{sec:freeproperad}
is this:
\begin{prop}
  The following diagram commutes.
  $$\xymatrix{
     \Fun(\elGr\op,\Grpd) \ar[r]^-{\ov{( \ )}}\ar[d]_{\pi_0 \circ -} & \Fun(\elGr\op,\Grpd) 
     \ar[d]^{\pi_0 \circ -} \\
     \Fun(\elGr\op,\Set) \ar[r]_-{\ov{( \ )}} & \Fun(\elGr\op,\Set) .
  }$$  
\end{prop}

\begin{proof}
  This follows from the fact that if a group $G$ acts on a groupoid $X$,
  then there is a natural bijection of sets
  $$
  \pi_0(X/\!\!/G) \simeq (\pi_0 X)/G , 
  $$
  where the left-hand side is the homotopy quotient and the right-hand side is the 
  naive quotient.
\end{proof}

This is particularly interesting if $F$ is the functor given by mapping into
a discrete (graph or) hypergraph $X$: in this case
$$
\Map(E,X)  = \pi_0 \Map(E,X) = \Hom(E,X),
$$
so that the commutativity of the square states that in this case the
free-properad monad construction of \ref{sec:freeproperad} actually factors
through the groupoid-enriched version.

\subsection{Free-properad monad on the category of hypergraphs}

Every hypergraph $X$ defines a bi-collection
\begin{eqnarray*}
  X : \elGr\op & \longrightarrow & \Grpd  \\
  E & \longmapsto & \Map_{\HGr}(E,X) ,
\end{eqnarray*}
and can therefore be given as argument to the free-properad monad.
The main result of this section states that
the free properad on a hypergraph is again a hypergraph.
The density lemma \ref{denshyp} is equivalent to:
\begin{cor}
  The natural functor $\HGr \to \PrSh(\elGr)=\Fun(\elGr\op,\Grpd)$ is fully faithful.
\end{cor}

We give the construction of the free-properad monad purely combinatorially,
within the category of hypergraphs, and check that it works.  In other words, we
define an endofunctor on $\HGr$, and check the commutativity of
$$\xymatrix{
   \HGr \ar[r]\ar[d] & \HGr \ar[d] \\
   \PrSh(\elGr) \ar[r] & \PrSh(\elGr) .
}$$
(Since the vertical maps are fully faithful, the top functor acquires monad
structure from the bottom functor.)

\begin{blanko}{Free-properad construction for hypergraphs.}\label{barhyp}
  Given a hypergraph $X$ (with constituents $AINOA$), let $\et(X)$ denote the
  groupoid of etale maps from a graph to $X$, i.e.~diagrams
    $$
\xymatrix{
\cdot \ar[d]& \ar[l] \cdot\ar[d] \ar[r]\drpullback  & \cdot \ar[d] & \dlpullback \ar[l]\cdot\ar[d] 
\ar[r] & \cdot\ar[d] \\
A & \ar[l] I \ar[r] & N & \ar[l] O \ar[r] & A
}
$$
where the first line is a graph.

  Let $\et^1(X)$ denote the groupoid of etale maps from a graph
  but with a marked import $e$.  Formally these are diagrams
  $$
  \xymatrix{
  1 \ar[d]_e& \ar[l] 0\ar[d] \drpullback\ar[r] & 0\ar[d] & \dlpullback \drpullback\ar[l]0\ar[d] 
  \ar[r] & 1\ar[d]^e \\
\cdot \ar[d]& \ar[l] \cdot\ar[d] \ar[r]\drpullback  & \cdot \ar[d] & \dlpullback \ar[l]\cdot\ar[d] 
\ar[r] & \cdot\ar[d] \\
  A & \ar[l] I \ar[r] & N & \ar[l] O \ar[r] & A
  }
  $$
  where the middle line is a graph.
  Similarly, $\et_1(X)$ is the groupoid etale maps from a graph
  but with a marked export.

  It is clear that these groupoids assemble into a diagram $\ov X$
  $$
  \xymatrix @! {
  A & \ar[l] \et^1 X \ar[r] & \et X & \ar[l] \et_1 X \ar[r] & A ,
  }
  $$
  where the structure maps delete appropriate rows of the diagrams representing
  $\et^1 X$ and $\et_1 X$.
\end{blanko}

\begin{lemma}
  If $X$ is a hypergraph, then so is $\ov X$, and this assignment
  is the object part of an endofunctor $\HGr \to \HGr$.
\end{lemma}

\begin{proof}
  We first establish that $\ov X$ is a hypergraph.  $A$ is discrete by
  assumption.  Let $p$ be a point in $\et(X)$, i.e.~an etale map $p:G\to X$, say
  of degree $d$, and where $G$ is a graph (and in particular is connected).  The
  vertex group $\Aut(p)$ is the group of
  deck transformations of the covering 
  $p$, and since $G$ is connected, it acts freely on the
  fibres.  In $\et^1(X)$, there are furthermore marked imports of $G$; the
  number of imports must be a multiple of $d$.  Let $p'$ denote the same
  covering but with a marked import.  Any such marked import must be
  fixed by the vertex group of $p'$, and therefore also any adjacent node must
  be fixed, but since $G$ is connected, it fixes all automorphisms, so
  $\Aut(p')$ is trivial.  Since $\Aut(p)$ acts freely on the fibres, it also
  acts freely on the set of imports, hence the homotopy quotient is again 
  discrete, so all the discreteness conditions
  are satisfied.  It remains to establish that $\et^1 X \to A
  \times \et (X)$ is a monomorphism.  So fix a hyperedge $e\in A$ and an etale map $G \to X$
  (element in $\et (X)$).  Well, that hyperedge either is or isn't an import
  of $G$, so the map is a monomorphism.  Similarly of course for exports.
  
  Finally for the functoriality: given an etale map $X \to Y$, we get maps $\et(X)
  \to \et(Y)$ by postcomposition, and similarly with the markings.  These
  form pullback squares, since for both hypergraphs the $\bar p$ and $\bar q$
  fibres over a subgraph $G$ are the sets of ports of $G$.
\end{proof}

\begin{lemma}\label{hyp:MapCXbar}
  Let $X$ be a hypergraph, and let $\ov X$ denote the hypergraph
  constructed in \ref{barhyp}.  Then the natural square
$$\xymatrix{
  \Map(C^m_n,\ov X) \ar[r] \ar[dd] & \et(X) \ar[d]^{\mathrm{dom}}\\
  & \Gr_{\iso} \ar[d]^{\res}\\
  1 \ar[r]_{\name{C^m_n}} & \kat{Cor} 
  }
  $$
  is a (homotopy) pullback.  Here $\res$ returns the corolla of ports of a graph,
  and the preceding map sends an etale map $G\to X$ to its domain.
\end{lemma}
\begin{proof}
  An object in this mapping space is a diagram
    $$\xymatrix{
     m+n \ar[d] & \ar[l] m \drpullback \ar[r]\ar[d] & 1 \ar[d] & \ar[l] 
     n
     \dlpullback \ar[d]\ar[r] & 
     m+n \ar[d] \\
     A  & \ar[l] \et^1(X) \ar[r] & \et(X) & \ar[l] 
     \et_1(X)\ar[r] & A
  }$$
  so the main ingredient is to give the middle vertical map, an element in
  $\et(X)$, i.e.~an etale map $G\to X$.  The fibre in $\et^1(X)$ over this
  element is naturally identified with the set of imports of $G$, so we need
  next to specify a bijection $m\isopil \operatorname{im}(G)$.  Similarly we
  need $n\isopil \operatorname{ex}(G)$.  In other words, we need to specify an
  isomorphism $C^m_n\isopil \res(G)$.  But this is also the description of the
  pullback.  It is easy to see that the arrows in the compared groupoids match
  up correctly as well.
\end{proof}

\begin{thm}\label{thm:HGr}
  The following diagram commutes.
  $$\xymatrix{
     \HGr \ar[r]^-{\ov{( \ )}}\ar[d] & \HGr \ar[d] \\
     \Fun(\elGr\op,\Grpd) \ar[r]_-{\ov{( \ )}} & \Fun(\elGr\op,\Grpd)
  }$$
\end{thm}

\begin{proof}
  Let $X$ be a hypergraph.
  We check that the two presheaves associated to $X$ agree on
  an elementary graph $C^m_n$.  For this consider the diagram
  of (homotopy) pullbacks of {\em groupoids}:
  $$
  \xymatrix{
  (m,n)\textrm{-et}(X) \drpullback \ar[r] \ar[d] & \et(X) \drpullback \ar[d] 
  \ar[r] & 1 \ar[d]^{\name{X}} \\
  (m,n)\kat{-Et} \drpullback \ar[d] \ar[r] & \Et \ar[d]^{\text{dom}} 
  \ar[r]_{\text{cod}} & \HGr \\
  (m,n)\kat{-Gr} \drpullback \ar[d] \ar[r] & \Gr \ar[d]^{\res} & \\
  1 \ar[r]_{\name{(m,n)}} & \kat{Cor} & 
  }
  $$
  Here we have suppressed notation to indicate that we are talking about {\em
  groupoids}, not categories.  Hence $\HGr$ denotes the groupoid of hypergraphs,
  and $\Gr$ denotes the groupoid of connected graphs.  Furthermore, $\kat{Et}= 
  \Gr \comma \HGr$
  denotes the groupoid whose objects are etale maps from a connected graph to a
  hypergraph, and whose morphisms are pairs of isos in the obvious way.  The
  upper left-hand corner is the groupoid whose objects are data $G \to X$ together
  with $(m,n) \isopil \res G$, consisting of an etale map from a graph into $X$,
  and a numbering of the ports of $G$.  Then the presheaf associated to the
  hypergraph $\ov X$ sends $(m,n)$ to the groupoid $\Map(C^m_n, \ov X)$.  By
  Lemma~\ref{hyp:MapCXbar}, this is precisely the upper left-hand corner of the big
  diagram.  On the other hand, if we apply the free-properad monad to the
  presheaf associated directly to $X$, the formula for the value on $(m,n)$ is
  $$
  \sum_{G\in \pi_0((m,n)\kat{-Gr})} \frac{\Map(G,X)}{\Aut_{(m,n)}(G)}
  $$
  (where the bar now denotes homotopy quotient).
  But this is precisely the upper left-hand corner of the diagram, expressed as
  a homotopy-sum of its fibres over objects in $(m,n)\kat{-Gr}$.
\end{proof}

\begin{BM}
  Under the `dual embedding' of (closed) directed graphs into hypergraphs
  \ref{dual}, the free-properad monad restricts to the free-category monad.
  (In this case, the groupoids involved are discrete.)
\end{BM}

\label{lastpage}

\end{document}